\RequirePackage[l2tabu, orthodox]{nag}
\documentclass[a4paper]{amsart}

\usepackage{amsmath,amsfonts,amssymb,amsthm,mathrsfs,calligra}
\usepackage{hyperref}
\usepackage{mathtools}
\usepackage{color}
\usepackage{enumerate}
\usepackage{framed}

\usepackage{microtype}
\usepackage{booktabs}
\usepackage{tabu}
\usepackage{multirow}

\newcommand{\func}[1]{\operatorname{#1}}
\newcommand{\ot}{\otimes}
\newcommand{\op}{\oplus}
\newcommand{\we}{\wedge}
\newcommand{\om}{\omega}

\newcommand{\G}{\mathcal{G}}
\newcommand{\R}{\mathbb{R}}
\newcommand{\C}{\mathbb{C}}
\newcommand{\g}{\mathfrak{g}}

\newcommand{\gsl}{\mathfrak{sl}}
\newcommand{\gsu}{\mathfrak{su}}
\newcommand{\m}{\mathfrak m}

\newcommand{\gh}{\mathfrak h}

\newcommand{\Ad}{\operatorname{Ad}}

\newcommand{\SO}{SO_{1,1}^+(\R)}
\newcommand{\sgn}{\operatorname{sgn}}
\newcommand{\Hom}{\operatorname{Hom}}
\newcommand{\gr}{\operatorname{gr}}

\newtheorem{thm}{Theorem}
\newtheorem*{thm*}{Theorem}
\newtheorem{prop}{Proposition}
\newtheorem{lem}{Lemma}
\newtheorem{cor}{Corollary}
\theoremstyle{definition}
\newtheorem{defi}{Definition}

\theoremstyle{remark}

\def\Z{\mathbb{Z}}

\begin{document}
\title{3-dimensional left-invariant sub-Lorentzian contact structures}
\author{Marek Grochowski}
\author{Alexandr Medvedev}
\author{Ben Warhurst}

\begin{abstract}
We provide a classification of $ts$-invariant sub-Lorentzian structures on $3$  dimensional contact Lie groups. Our approach is based on invariants arising form the construction of a normal Cartan connection.
\end{abstract}

\maketitle

\section{Introduction}

\subsection{Sub-Lorentzian Geometry}

Let $M$ be a smooth manifold. \textit{A sub-Lorentzian structure} (or 
\textit{a metric}) on $M$ is, by definition, a pair $(H,g)$ where $H$ is a
smooth bracket generating distribution on $M$ and $g$ is a smooth Lorentzian
metric on $H$. A triple $(M,H,g)$ where $M$ is a manifold endowed with a sub-Lorentzian structure $(H,g)$ is called \textit{a sub-Lorentzian manifold}. Such a manifold is said to be contact if the distribution $H$ is contact. 

The theory of sub-Lorentzian manifolds is a subject presently in it's infancy but interest is growing, see \cite{Gr1, Gr2, Gr3} for outlines of the general theory and \cite{KM, Vas, Gr5} where the main emphasis was on investigating the causal structure of such manifolds (geodesics, the structure of reachable sets, normal forms). There are also two papers \cite{BM, GK1} where the authors started 
the investigation of the theory of invariants for sub-Lorentzian 
structures.

Sub-Lorentzian metrics which are simultaneously time and space oriented will be called $ts$-oriented (see section 2 for more details). The aim of this paper is to classify all 
left invariant $ts$-oriented 
sub-Lorentzian structures on $3$-dimensional Lie groups. Two basic 
invariants which plays fundamental role in the study were introduced in 
\cite{BM}. These are: a $(1,1)$-tensor $h=\begin{psmallmatrix}
a & b\\ -b & -a 
\end{psmallmatrix}$ 
on $H$ and a smooth function $\kappa $ on $M$ (see Section 
\ref{sec:Inv} 
for exact formulas). These invariants can be obtained from 
a variety of perspectives with varying degrees of complexity. For 
example in \cite{GK1} a Riemannian approach is used while an approach 
using null lines should also be possible. In both cases, the fact that 
the underlying manifold is contact plays a significant role and so the 
techniques do not obviously generalize. 

In this work we use the standard machinery of Cartan connections with 
the goal 
of giving a thorough explanation of the origins of the invariants via Cartan 
theory. The general approach proceeds as follows: to begin, we construct the 
so called first order geometric structure corresponding to the sub-Lorentzian
structure $(M,H,g)$ (see section \ref{sec21} for the definitions). Specifically, 
we construct the symbol algebra corresponding to a canonical filtration of the 
tangent bundle generated by $H$.  The symbol algebra then defines the weighted 
frame bundle.  The first order geometric structure is an $SO_{1,1}$ reduction of 
the bundle of weighted frames. A frame on $M$ is said to be adapted if it corresponds
in a natural manner with a weighted frame and the bundle over $M$ consisting of adapted frames
admits a canonical Cartan connection. The curvature function of this connection generates the algebra of invariants for the corresponding 
sub-Lorentzian structure. In particular the invariants $\kappa$ and $h$ arise in a canonical decomposition 
of the curvature function for 3-dimensional contact sub-Lorentzian 
structure.

One can go further and obtain canonical frames using the $\SO$-equivariance 
of the structure function. This leads to the full classification of 
3-dimensional left-invariant sub-Lorentzian contact structures. In the 
following theorem we use the notation of \v{S}nobl and Winternitz 
\cite{SnobWint} for 3-dimensional real Lie algebras. In particular, 
$L(3,1)=\gh_3$ is the Heisenberg Lie algebra; the Lie
algebras $L(3,2,\eta)$, $L(3,4,\eta)$ and $L(3,3)$ are solvable with the
special cases $L(3,2,-1)=\mathfrak{p}_{1,1}$ and $L(3,4,0)=\mathfrak{e}_2$ being  the Poincar\'e and Euclidean
Lie algebras respectively; $L(3,5)$ and $L(3,6)$ are $\gsl_2$ and 
$\gsu_2$ respectively.
 Our main results is:
\begin{thm}\label{mainThm}
 All 3-dimensional left-invariant sub-Lorentzian contact structures up 
 to local $ts$-isometries are given by Table \ref{mainTbl}, where each 
 connected simply connected Lie group is associated with it's Lie 
 algebra. 
 Each row in the  table corresponds to a different structure. In 
 particular
 \begin{enumerate}
 	\item If $\det h\le 0$ then the structure is completely 
 	determined by the invariants $h$, $\kappa$ and $\tau$ if $\tau$ is 
 	defined.
 	\item If $\det h > 0$ then every value of $\kappa$ and $h$ admits 3 
 	locally non-isomorphic structures.
 \end{enumerate}
\end{thm}

\begin{table}[h]
	\caption{3D left-invariant contact sub-Lorentzian structures}
	\label{mainTbl}
	\begin{tabu}{ *4{>{$} c<{$}} }
		\toprule
		h \mbox{ equivalent to}
		& \kappa 
		& \tau, \mbox{ if defined}
		& \mbox{Lie Algebra} 
		\\ 
		\midrule
		& \kappa =0 
		& -
		& L(3,1) = \gh_3
		\\
		\multirow{-2}{*}{
			$\begin{pmatrix} 0 & 0 \\ 0 & 0 \end{pmatrix}$ } 
		& \kappa\in \R^*
		& -
		& L(3,5)=\gsl_2
		\\
		\midrule
		& \kappa =0 
		& \tau=2
		& L(3,3)
		\\
		& \kappa=0 
		& |\tau|>2
		& 
		L\left(3,2,\frac{-\tau-\sqrt{\tau^2-4}}{-\tau+\sqrt{\tau^2-4}}\right)
		\\
		& \kappa=0 
		& |\tau|<2
		& L\left(3,4,\frac{|\tau|}{\sqrt{\tau^2-4}}\right)
		\\
		\multirow{-5}{*}{
			$\begin{pmatrix} 1 & -1 \\ 1 & -1 \end{pmatrix}$, 
			$\begin{pmatrix} -1 & -1 \\ 1 & 1 \end{pmatrix}$  } 
		& \kappa\in \R^*
		& -
		& L(3,5)=\gsl_2
		\\
		\midrule
		& \kappa=0 
		& \tau\in \R
		& 
		L\left(3,2,\frac{\tau-\sqrt{\tau^2+4}}{\tau+\sqrt{\tau^2+4}}\right)
		\\
		\multirow{-2}{*}{
			$\begin{pmatrix} 1 & 1 \\ -1 & -1 \end{pmatrix}$, 
			$\begin{pmatrix} -1 & 1 \\ -1 & 1 \end{pmatrix}$  } 
		& \kappa\in \R^* 
		& -
		& L(3,5)=\gsl_2
		\\
		\midrule
		& |\kappa|<-\chi
		& -
		& L(3,6)=\gsu_2
		\\
		\multirow{-2}{*}{
			$\begin{pmatrix} 0 & \chi \\ -\chi & 0 \end{pmatrix}$, $\chi\neq 
			0$ } 
		& |\kappa|>-\chi,|\kappa|\neq\chi
		& -
		& L(3,5)=\gsl_2
		\\
		& \chi=\pm\kappa>0,
		& -
		& L(3,2,-1)=\mathfrak{p}_{1,1}
		\\
		& \chi=\pm\kappa<0,
		& -
		& L(3,4,0)=\mathfrak{e}_2
		\\
		\midrule
		& \kappa=-7\chi
		& -
		& L(3,3)
		\\
		\multirow{-2}{*}{
			$\begin{pmatrix} 0 & \chi \\ -\chi & 0 \end{pmatrix}$, $\chi\neq 
			0$ } 
		& \kappa>-7\chi
		& -
		& L\left(3,2,\frac{\sqrt{|\chi-\kappa|}-\sqrt{7\chi+\kappa}}
		{\sqrt{|\chi-\kappa|}+\sqrt{7\chi+\kappa}}\right)
		\\
		& \kappa<-7\chi,
		& -
		& L\left(3,4,\frac{\sqrt{|\chi-\kappa|}}
		{\sqrt{-7\chi-\kappa}}\right)
		\\
		\midrule
		& \kappa=7\chi
		& -
		& L(3,3)
		\\
		\multirow{-2}{*}{
			$\begin{pmatrix} 0 & \chi \\ -\chi & 0 \end{pmatrix}$, $\chi\neq 
			0$ } 
		& \kappa<7\chi
		& -
		& L\left(3,2,\frac{\sqrt{|\chi+\kappa|}-\sqrt{7\chi-\kappa}}
		{\sqrt{|\chi+\kappa|}+\sqrt{7\chi-\kappa}}\right)
		\\
		& \kappa>7\chi,
		& -
		& L\left(3,4,\frac{\sqrt{|\chi+\kappa|}}
		{\sqrt{-7\chi+\kappa}}\right)
		\\
		\midrule
		\begin{pmatrix} \chi & 0 \\ 0 & -\chi \end{pmatrix}$, $\chi\neq 0
		& \kappa \in\R
		& -
		& L(3,5)=\gsl_2
		\\
		\bottomrule
	\end{tabu}
\end{table}

We outline the main similarities and differences with the sub-Riemannian 
case. First of all, as a byproduct of our classification procedure, we see specific similarities with phenomena which arise in \cite{Agr1}, namely that the affine group $A_1(\R)\op\R$ is locally $ts$-isometric to $SL_2(R)$. Secondly, the 
scalar invariants are similar to the sub-Riemannian case with the differences simply in sign. However, the tensor $h$ represents a significant difference. In particular we have 4 cases to consider: 
$h=0$; $\det h=0,h\neq 0$; $\det h>0$ and $\det h<0.$ Finally, 
the case $\det h=0,h\neq 0$, $\kappa=0$ has a non-discrete family of 
non-equivalent structures. We introduced the invariant $\tau$, which is in 
fact the covariant derivative of $h$, to parametrise this family.

The paper is organized as follows. In section 2 we make precise some terminology concerning orientation and isometries as well as provide a motivating example comming from control theory. In section 3 we outline Cartan geometry and the construction of canonical Cartan connections. In section 4 we apply the Cartan methods to our specific case to produce the invariants of the sub-Lorentzian srtuctures of interest. In section 5 we combine the invariants with the classsification procedure of \v{S}nobl and Winternitz \cite{SnobWint} to prove our main theorem \ref{mainThm}. 

\section{Sub-Lorentzian preliminaries}

\subsection{Orientation} Suppose that $(M,H,g)$ is a sub-Lorentzian manifold and let $q\in M$. A vector $v\in H_{q}$ is said to be timelike if $g(v,v)<0$, null if $g(v,v)=0$ and $v\neq 0$, nonspacelike if $g(v,v)\leq 0$, $v\neq 0$, and finally spacelike if $g(v,v)>0$ or $v=0$. A vector field on $M$ is called timelike (nonspacelike, null) if its values $X(q)$ have such a property for every $q\in M$. The definition of course implies that any timelike, nonspacelike or null vector field is horizontal in the sense that it is a section of the bundle $H\to M$.

If $(H,g)$ is a sub-Lorentzian metric on $M$, then it can be proved that $H$ admits a splitting $H=H^{-}\oplus H^{+}$ of sub-bunbdles such that $H^{-}$ is of rank $1$ and $g$ is negative (positive) definite on $H^{-}$ ($H^{+}$). Any splitting with the above-mentioned properties is called a causal decomposition for $(H,g)$. We say that the metric $(H,g)$ is time-orientable (resp. space-orientable) if the bundle $H^{-}\to M$ (resp. $H^{+}\to M$) is orientable. Consequently, by a time (space) orientation we mean a given orientation of the bundle $H^{-}\to M$ ($H^{+}\to M$). Note that since a rank $1$ bundle is orientable if and only if it is trivial, time orientability of $(M,H,g)$ is equivalent to the existence of a timelike vector field on $M$. Thus it is sometimes more convenient to define a time orientation of $(M,H,g)$ as a choice of a timelike vector field $X$ on $M$.

Since causal decompositions are not unique we must make it precise when two
pairs of bundles $H_{1}^{+}\to M$, $H_{2}^{+}\to M$
and $H_{1}^{-}\to M$, $H_{2}^{-}\to M$ have
compatible orientations, where $H=H_{1}^{-}\oplus H_{1}^{+}$and $%
H=H_{2}^{-}\oplus H_{2}^{+}$ are causal decompositions. So we say that the
given orientations of $H_{1}^{-}\to M$, $H_{2}^{-}%
\to M$ are compatible if $g(X_{1},X_{2})<0$, where $X_{i}$ is
the section of $H_{i}^{-}\to M$ which agrees with the
orientation of $H_{i}^{-}\to M$, $i=1,2$. On the other hand, $%
H_{1}^{+}\to M$, $H_{2}^{+}\to M$ have compatible
orientations if the following condition is satisfied: for any point $q\in M$
there exists its neighbourhood $U$ and linearly independent sections $%
X_{i,1},...,X_{i,k}:U\to H_{i}^{+}$, $rank$ $H_{i}^{+}=k$, such
that $X_{i,1},...,X_{i,k}$ agrees with the orientation of $%
H_{i}^{+}\to M$, $i=1,2$, and $\det \left(
g(X_{1,i},X_{1,j})\right) _{i,j=1,...,k}>0$.

As mentioned earlier, Sub-Lorentzian metrics which are simultaneously time and space oriented will be called $ts$-oriented. If $(M,H,g)$ is time-oriented by a vector field $X$ then a nonspacelike $v\in H_{q}$ is said to be future directed if $g(v,X(q))<0$.

\subsection{Isometries} Let $(M,H,g)$ be a sub-Lorentzian manifold. A diffeomorphism $%
f:M\to M$ is called an isometry if (i) $d_{q}f(H_{q})=H_{f(q)}$
for every $q\in M,$ and (ii) $d_{q}f:H_{q}\to H_{f(q)}$ is a
linear isometry for every $q$, that is to say if $%
g(d_{q}f(v),d_{q}f(w))=g(v,w)$ for all $v,w \in H_{q}$ and $q\in M$. Of
course, all isometries of a given sub-Lorentzian manifold form a group. Let $%
f:M\to M$ be an isometry and let $H=H^{-}\oplus H^{+}$ be a
causal decomposition. Then $H=H_{1}^{-}\oplus H_{1}^{+}$ is also a causal
decomposition, where $H_{1}^{\pm }=df(H^{\pm })$. Suppose now that $(M,H,g)$
is time-oriented (resp. space-oriented). We say that $f$ is a $t$-isometry
(resp. $s$-isometry) if the fiber bundle map $df|_{H^{-}}:H^{-} \to H_{1}^{-}$ (resp. $df|_{H^{+}}:H^{+} \to H_{1}^{+} $) is orientation preserving. In case $(M,H,g)$ is $ts$-oriented, $f$ is called a $ts$-isometry if it preserves both orientations. Clearly, any 
$ts$-isometry $f$ is characterized by the condition $d_{q}f\in SO_{1,1}(k)$
for every $q\in M$.

\subsection{An application}
Sub-Lorentzian manifolds arise in control theory. Suppose that $(M,H,g)$ is a time-oriented sub-Lorentzian manifold, $\mathrm{rank}\,H=k+1$. Let $X_0,X_1,\dots,X_k$ be an orthonormal basis for $(H,g)$, defined
on an open set $U\subset M$, where $X_0$ is a time orientation. An absolutely continious curve $\gamma\colon (a,b)\to M$ is said to be 
\textit{nonspacelike future directed}, if $\dot{\gamma}(t)\in H_{\gamma(t)}$, $g(\dot{\gamma}(t),X_0(\gamma(t))) \leq 0$ and $g(\dot{\gamma}(t),\dot{\gamma}(t))\leq 0$ for almost every $t$ in $(a,b)$.
It can be proved (see \cite{Gr5}) that up to a reparameterization all nonspacelike future directed curves in $U$ can be obtained as trajectories of the following 
affine control system on $U$:
\begin{align}\label{AffCon}
 \dot{q}=X_0+\sum_{i=1}^{k}u_iX_i \text{,\;\;}
\end{align}
where the set of control parameters is equal to
\begin{align*}
 \mathcal{C}=\left\{u\in\R^{k}\text{;\;} \sum_{i=1}^{k} u_i^2\leq 1\right\} \text{,}
\end{align*}
i.e. to the  unit ball centered at $0$. Note that (\ref{AffCon}) is not uniquely determined by the structure $(H|_U,g)$, it depends on the choice of an
orthonormal frame $X_0,X_1,\dots,X_k$, where $X_0$ is a time orientation. However, any two such systems are equivalent in the sense that they have the same set of trajectories.

Affine control systems frequently arise in various fields of mathematics and physics (cf. ,\cite{Jurd}, \cite{Mont}), so it is worth noting that our results can be used to classify such systems in the case where $M$ is a $3$ dimensional Lie group, $g$ is left-invariant and $k=1$. For more general systems, Cartan geometry applies equally well, and so the outline given here can also be used in such cases.  

\section{Cartan Geometries Associated with Structures on Filtered 
Manifolds}

\subsection{Geometric structures on filtered manifolds} \label{sec21}
A filtered manifold is basically a manifold endowed with a filtration of the tangent bundle. A manifold with a bracket generating distribution endowed with a (pseudo) sub-Riemannian structure with constant symbol sits within the purview of the general theory of geometric structures on filtered manifolds.
 In this section we briefly describe some basic constructions from 
 nilpotent differential geometry \cite{Morimoto93}. Our aim is to 
 illustrate a general strategy for the construction of canonical Cartan 
 connections related to geometries on filtered manifolds.

\begin{defi}
	A filtration of the tangent bundle of a manifold $M$ is a sequence 
	$\{F^i\}_{ i \in \Z}$ of sub-bundles of $TM$ such that 
	\begin{enumerate}[i]
		\item $F^0 =\{0\}$,
		\item $F^{i+1} \subset F^i$,
		\item $\cup_{i \in \Z} F^i=TM$,
		\item $[\underline{F}^i, \underline{F}^j] \subset 
		\underline{F}^{i+j}, \quad \forall i,j \in \Z,  $
	\end{enumerate}
	where $\underline{F}^i$ denotes the sheaf of germs of local sections 
	of $F^i$. A \emph{filtered manifold} is a smooth manifold $M$ 
	equipped with a filtration of the tangent bundle. 
\end{defi}

A filtration naturally arises from a smooth bracket generating 
distribution $D$. Set $\underline{F}^i=\{0\}$ for all $i \geq 0$ and 
let $ \underline{F}^{-1}$ denote the sheaf of germs of local sections 
of $D$. For $i<0$, we inductively define a sequence of sheaves by 
setting $\underline{F}^{i-1} = \underline{F}^{i} +[\underline{F}^i 
,\underline{F}^{-1} ].$ The filtration is the sequence $\{F^i\}$ 
where each $F^i$ is the union  of stalks $F^i=\cup_{p \in M} 
\underline{F}^i|_p $.

If $ \func{gr}_{-i} T_pM=F^{-i}_p /F^{-i+1}_p$, then the graded tangent 
space is the vector space 
\[\func{gr} T_p M = \bigoplus_{i=1}^k \func{gr}_{-i} T_pM.\]
If $X \in \Gamma(F^{-i})$ and $Y \in \Gamma(F^{-j})$ are local sections 
defined on a neighborhood of $p$, then we 
have 
\[ [  X+ \Gamma(F^{-i+1}) ,  Y+ \Gamma(F^{-j+1})] = [ X, Y] + 
\Gamma(F^{-(i+j)+1}) \]
on some neighborhood of $p$, and it follows that the Lie bracket of 
vector fields  induces a well defined Lie bracket on $\func{gr} T_p M$ 
thus defining a stratified nilpotent Lie 
algebra of step $k$. Lie algebra  $\func{gr} T_p M$ is often called  \emph{a symbol Lie algebra} of the filtration $F$ in the point $p.$ We say that a filtered manifold is \emph{of type $\mathfrak{m}$} if $\func{gr} T_p M$ equipped with the induced Lie bracket is 
isomorphic with $\mathfrak{m}$ as a  stratified nilpotent Lie algebra at every point $p$.

\begin{defi}
Let $U$ be an open subset of a filtered manifold $M$ of type $\mathfrak{m}$. \emph{A local weighted frame} is a map $\Phi:U \times \mathfrak{m} \to \gr TM$ such that for each point $p \in U$ the map $\Phi_p: \mathfrak{m} \to \func{gr} T_pM$ is a strata preserving Lie algebra isomorphism. The union of germs of local weighted frames over all points of $M$ defines an $\func{Aut}(\mathfrak{m})$-principle bundle which we refer to as the \emph{weighted frame bundle} of the filtered manifold $M$ of type $\mathfrak{m}$.
\end{defi}

 \begin{defi} If $M$ is filtered of type $\mathfrak{m}$ then a local 
 frame $\{E_j^{-i}\}$ of $TM$ is said to be \emph{adapted} if there 
 exists a weighted frame  $\Phi$  such that $\Phi(p,e_j^{-i}) = 
 E_j^{-i}(p) +F^{-i+1}_p$ where $e_j^{-i}$ form a basis of 
 $\mathfrak{m}_{-i}$. If $\{\omega^j_{-i}\} \subset T^*M$ is the set of 
 dual elements to $\{E_j^{-i}\}$, then the map $ e_j^{-i} \otimes 
 \omega^j_{-i} \colon TM\to \mathfrak{m}$  is called \emph{an adapted 
 coframe}.	
\end{defi}

\begin{defi}
	A \emph{first order geometric structure} on a filtered manifold of 
	type 
	$\mathfrak{m}$
	is a reduction of the weighted frame bundle to the subgroup $G_0 
	\subset \func{Aut}(\mathfrak{m})$. We call a weighted frame which 
	belongs to the reduced bundle \emph{a weighted frame adapted to the 
		structure.}
\end{defi}

Consider a (pseudo) sub-Riemannian structure which is defined by a 
pseudo-Riemannian metric $g$ on the distribution $\mathcal{H}\subset 
TM$. If the distribution is bracket generating then $M$ is a filtered 
manifold and the metric $g$ defines a metric (pseudo) sub-Riemannian 
symbol $(g,\m)$ in every point. If metric symbols are equivalent for 
all points we say that the corresponding structure has constant metric 
symbol. For such structures the weighted frame bundle can be reduced to 
the structure group $G_0\subseteq SO(\m_{-1})$ which gives us a 
canonically defined first order geometric structure on $M$.

\subsection{Cartan geometries as a generalization of homogeneous spaces 
}

A homogeneous space for a Lie group $G$ is a manifold $M$ on which the 
Lie group $G$ acts from the right both transitively and effectively. 
If $H$ is the stabilizer of an arbitrary point $p\in M$ 
then $H$ is a closed Lie subgroup of $G$. The manifold $M$ is 
diffeomorphic to the right coset space $G/H$ under the action $(Ha)g = 
H(ag)$.

Consider the 3 natural actions which are defined for all Lie groups $G$:
\begin{itemize}
\item left multiplication $L_g(a)=ga$,

\item right multiplication $R_g(a)=ag$,

\item conjugation $C_g(a)=gag^{-1}$,
\end{itemize}
All maps above are diffeomorphisms and their differentials are denoted by $%
L_{g*}$, $R_{g*}$ and $Ad_g$ respectively. In particular, the tangent map 
\begin{equation*}
L_{g^{-1}*}\colon T_g G \to T_e G 
\end{equation*}
is an isomorphisms of tangent spaces. The \emph{Maurer-Cartan form} 
$\tilde\omega
\colon TG \to \mathfrak{g}$ is defined pointwise by 
$\tilde\omega_g=L_{g^{-1}*}$
and satisfies the following 3 properties which will be crucial in the
definition of Cartan connection:

\begin{enumerate}
\item ${\tilde\omega}_g \colon T_g G \to \mathfrak{g}$ is an 
isomorphism;

\item $R_g^*\tilde\omega = Ad_{g^{-1}} \tilde\omega ;$

\item for all left-invariant vector field $X$ we have 
$\tilde\omega_g(X_g)=X_e$.
\end{enumerate}

We introduce the commutator on 
forms with values in a Lie algebra $\mathfrak{g}$ defined by
\begin{equation*}
[\alpha,\beta] (X,Y)=[\alpha(X),\beta(Y)] + [\beta(X),\alpha(Y) ].
\end{equation*}
In particular
\begin{equation*}
[\alpha(X),\alpha(Y)] = \frac12 [\alpha,\alpha] (X,Y).
\end{equation*}
If $e_i$ is basis of $\mathfrak{g}$, $\alpha=\sum_i e_i \otimes 
\alpha_i$ and $\beta=\sum_i e_i \otimes \beta_i$, then $[\alpha,\beta]$ 
is defined by 
\begin{equation*}
[\alpha,\beta]=\sum_{i,j} [e_i,e_j] \otimes \alpha_i\wedge\beta_j .
\end{equation*}

One of the key properties of the Maurer-Cartan form $\tilde\omega$ is 
that the following
structure equation holds: 
\begin{equation*}
d\tilde\omega +\frac12[\tilde\omega ,\tilde\omega ]=0. 
\end{equation*}
To show that structure equation holds, it is sufficient to check it for
left-invariant vector fields. Using Cartan's formula we obtain 
\begin{align*}
d\tilde\omega(X,Y)&=X\tilde\omega(Y)-Y\tilde\omega(X)-\tilde\omega([X,Y])
 \\
& = -\tilde\omega([X,Y]) =-[\tilde\omega(X),\tilde\omega(Y) ] = 
-\frac12 [\tilde\omega,\tilde\omega](X,Y),
\end{align*}
where $X$ and $Y$ are left-invariant vector fields.

Cartan geometries generalize homogeneous spaces $G \to G/H$ simply 
by considering a general principal $H$-bundle $\mathcal{G} \to \mathcal{G}/H$ and 
prescribing an object on $\mathcal{G}$ analogous to the Maurer Cartan form on $G$, 
where properties analogous to (2) and (3) above are only required to 
hold with respect to $H$.

Let $\mathfrak{g}$ and $\mathfrak{h}$ denote the Lie algebras of the Lie groups $G$ and $H$ respectively.

\begin{defi}
A \emph{Cartan geometry} of infinitesimal type $(\mathfrak{g}, 
\mathfrak{h})$ on a manifold $M$ is a principal $H$-bundle 
$\mathcal{G}$ over $M$ together with a form $\tilde 
\omega:T\mathcal{G}\to \mathfrak{g}$, called the \emph{Cartan 
connection form}, having the following properties:
\begin{enumerate}
\item ${\tilde \omega}_p\colon T_p\mathcal{G} \to \mathfrak{g}$ is an isomorphism;
\item $R_h^* \tilde \omega = \func{Ad}_{h^{-1}} \tilde \omega $ for all $h\in H$;
\item $\tilde \omega(X^*)=X$ where $X^*$ is a fundamental vector field 
corresponding to $X \in \mathfrak{h}$, i.e., $X^*f(p)= \frac{d}{dt} f(p 
\exp(tX))|_{t=0}$.
\end{enumerate}
\end{defi}

The Maurer-Cartan structure equation doesn't hold for general Cartan
connections. The $\mathfrak{g}$-valued 2-form $\tilde \Omega \in \mathfrak{g}
\, \otimes \bigwedge^2 T\mathcal{G}^*$ given by the formula 
\begin{equation*}
\tilde \Omega=d \tilde \omega+ \frac12 [\tilde \omega,\tilde \omega], 
\end{equation*}
is called the \emph{curvature form}. A fundamental property of the curvature
is that $v_p \, \lrcorner \, \tilde \Omega_p=0$ for all $v_p$ belonging to the
vertical sub-bundle $\mathcal{V}=\mathrm{ker}\, \pi_*$, where $\pi : \mathcal{%
G} \to M$ is the natural projection, see \cite[p. 187]{Sharpe}.

On the manifold $M$ the relevant object to study is the pull back of the Cartan connection by a section of the principle bundle $\mathcal{G}$.

\begin{defi}
Given an arbitrary section $s\colon M\to\mathcal{G}$, the \emph{Cartan 
gauge} corresponding to $s$ is the one form $\omega=s^*\tilde{\omega}$.
\end{defi}

Consider a change of section $\bar s= s h,$ where $h\colon M\to H$. Then the Cartan gauge changes in the following way: 
\begin{equation}  \label{intro:cartan_change}
\bar\omega =\bar s^*\tilde{\omega} =\func{Ad}_{h^{-1}}\omega +
h^*\omega_H=Ad_{h^{-1}}\omega+ h^{-1} dh,
\end{equation}
where $\omega_H$ is Maurer-Cartan form on $H$.

The pull-back of the curvature $\tilde{\Omega}$ on the principle $H$-bundle $\mathcal{G}$ is a two form $\Omega=s^*\tilde{\Omega}$ on the manifold $M$ and is given by the following formula 
\begin{equation*}
\Omega=d\omega+\frac12 [\omega,\omega].
\end{equation*}
If we change section $\bar s= s h$ the curvature on the manifold changes by the adjoint action of $h^{-1}$: 
\begin{equation}
\bar\Omega =\bar s^*\tilde{\Omega} =\func{Ad}_{h^{-1}}\Omega.
\end{equation}

\subsection{Cartan connections associated with structures on filtered manifold}

The problem of equivalence between geometric structures on manifolds is typically solved by applying Cartan's method of equivalence to produce a Cartan connection and use its curvature as the natural invariant. 

When the underlying manifolds are filtered, the target Lie algebra for 
a Cartan connection is given by the Tanaka prolongation of the pair 
$(\mathfrak{m},\mathfrak{g}_0)$ where $\mathfrak{g}_0$ is a subalgebra 
of the strata preserving derivations of $\mathfrak{m}$ such that 
$G_0=\exp(\mathfrak{g}_0)$, see \cite{Tanaka1}. 

Consider a graded nilpotent Lie algebra $\mathfrak{m}=\mathfrak{m}_{-k}
\oplus \dots \oplus \mathfrak{m}_{-1} $. Let $\mathfrak{g}_0$ be a 
subalgebra of the grading preserving derivations of $\mathfrak{m}$. The 
Tanaka prolongation of the pair $(\mathfrak{m},\mathfrak{g}_0)$ is the 
graded Lie algebra $\mathfrak{g}(\mathfrak{m},\mathfrak{g}_0)$ where 
$\mathfrak{g}_i(\mathfrak{m},\mathfrak{g}_0)=\mathfrak{m}_i$ for $-k 
\leq i < 0$, 
$\mathfrak{g}_0(\mathfrak{m},\mathfrak{g}_0)=\mathfrak{g}_0$ and for 
each $i > 0$, $\mathfrak{g}_i(\mathfrak{m},\mathfrak{g}_0)$ is 
inductively defined by 
\begin{equation*}
\mathfrak{g}_i (\mathfrak{m},\mathfrak{g}_0)= \Big \{ \varphi \in
\bigoplus_{p>0} \mathfrak{g}_{i-p}(\mathfrak{m},\mathfrak{g}_0) \otimes 
\mathfrak{g}_{-p}^* \ | \ \varphi([X,Y])=[\varphi(X),Y]+[X,\varphi(Y)]
\Big \}.
\end{equation*}
The pair $(\mathfrak{m},\mathfrak{g}_0)$ is said to be of \emph{finite 
type} if $\mathfrak{g}_i(\mathfrak{m},\mathfrak{g}_0)=\{0\}$ for some 
$i$, otherwise it is of \emph{infinite type} and 
$\mathfrak{g}(\mathfrak{m},\mathfrak{g}_0)$ is infinite dimensional.

Consider a first order geometric structure of type $\mathfrak{m}$ on 
the filtered manifold $M
$. Let $\mathfrak{g} = \mathfrak{g}(\mathfrak{m},\mathfrak{g}_0) $ and $%
\mathfrak{g}_{+}=\oplus_{i>0} \mathfrak{g}_{i}$. We define a trivial $H$%
-principal bundle $\mathcal{G} = H\times M$ where 
\begin{equation*}
H=G_0\times \exp(\mathfrak{g}_+).
\end{equation*}
With the given first order geometric structure we associate a family of 
adapted Cartan
connections of type $(\mathfrak{g},\mathfrak{h})$.

\begin{defi}
A Cartan connection $\tilde \omega\colon T\mathcal{G} \to \mathfrak{g}$ is called \emph{adapted} if for an arbitrary section $s\colon M\to \mathcal{G}$, the corresponding  Cartan gauge $\omega \colon TM\to \mathfrak{g}$, has the property that the $\mathfrak{m}$ valued part forms an adapted coframe.
\end{defi}

To obtain invariants of the initial geometric structure we want to associate
a unique adapted Cartan connection to the structure. The construction 
of the desired connection can be done using normalization of the 
structure function.

\begin{defi}
The \emph{curvature function} $\tilde k\colon \mathcal{G} \to 
\func{Hom}(\wedge^2 
\mathfrak{g}_-,\mathfrak{g})$ of a Cartan connection $\tilde{\omega%
}$ is defined by the formula 
\begin{equation*}
\tilde k(X,Y) = \tilde{\Omega} \left(\tilde \omega^{-1}(X),\tilde 
\omega^{-1}(Y)\right).
\end{equation*}
\end{defi}

A built in property of the curvature function is that it is $H$-equivariant, 
i.e., 
\begin{equation*}
R_{h}^* \tilde k = Ad_{h^{-1}} \tilde k . 
\end{equation*}

The vector space $ \func{Hom}(\wedge^2 \mathfrak{g}_-,\mathfrak{g}) = \g \otimes \bigwedge^2 \m^* 
$ has a natural grading. Elements in the subspace $\g_{l}\otimes \m_{-i}^* \wedge \m_{-j}^*$ are 
assigned weight $w=l+i+j$.

\begin{defi}
	We call a Cartan connection and underlying Cartan geometry 
	\emph{regular} if 
$\tilde k$ takes values in $\func{Hom}(\wedge^2 
		\mathfrak{g}_-,\mathfrak{g})_{+}$, where $+$ 
		subscript means a 
		positive degree part of a space.
\end{defi}
\begin{defi}
A subspace $N \subset \func{Hom}(\wedge^2 
\mathfrak{g}_-,\mathfrak{g})_+$ is called a normal module if:
\begin{enumerate}
\item $N$ is an $H$ module with respect to the adjoint action of $H$ on 
$\func{Hom}(\wedge^2 \mathfrak{g}_-,\mathfrak{g})$;
\item $\func{Hom}(\wedge^2  
\mathfrak{g}_-,\mathfrak{g})_+=N\oplus\partial 
\left(\func{Hom}(\mathfrak{g}_-,\mathfrak{g} )_+\right)$, where 
$\partial$ is the Lie algebra differential.
\end{enumerate}
\end{defi}

The existence of unique Cartan connection associated with a geometric 
structure is a fundamental starting point in the study of equivalence 
problems for the given geometric structure. For example regular 
parabolic geometries (i.e. with semi-simple model group) admit a 
natural and uniform notion of normal Cartan connection 
\cite{CapSlovak}. The general result regarding existence of normal 
Cartan connections can be formulated as follows:

\begin{thm}\label{prop:morimoto}
{\cite[p. 92]{Morimoto93}} Consider a geometric structure with an infinitesimal model $(\mathfrak{g},\mathfrak{h}) $ on a filtered manifold. Then for every normal module 
\begin{equation*}
N \subset \func{Hom}(\wedge^2 \mathfrak{g}_-,\mathfrak{g})_+
\end{equation*}
there exists a unique regular Cartan connection adapted to the structure such that the curvature function takes values in $N$.
\end{thm}

In the next section we are going to construct various invariant objects 
associated with sub-Lorentzian structures on a contact $3$-manifold 
$M$. The main ingredients are a normal module and the corresponding 
normal Cartan connection given by Theorem \ref{prop:morimoto}. A 
canonical pullback of the Cartan connection to $M$ induces 
differential invariants for the structures at the level of $M$ and 
allows us to construct canonical frames for sub-Lorentzian contact 
structures.

\section{Invariants of 3-dimensional sub-Lorentzian contact structures}

\subsection{First order geometric structures associated with  
3-dimensional sub-Lorentzian contact structures}
The sub-Lorentzian contact structure is given by a 
contact distribution $\mathcal{H}$ and a sub-Lorentzian metric $g$ 
which is defined on $\mathcal{H}$. Let $X_{1}$ and $X_{2}$ be an 
orthogonal frame of $\mathcal{H}$ in the following sense: 
\begin{equation*}
g(X_{1},X_{1})=-1,\,g(X_{1},X_{2})=0,\,g(X_{2},X_{2})=1.
\end{equation*}
We choose a contact form $\eta $ so that $d\eta 
(X_{1},X_{2})=\eta([X_{2},X_{1}])=1$ and denote the corresponding Reeb 
vector field by $X_{3}$, i.e., $X_{3}\lrcorner d\eta =0$ and $\eta 
(X_{3})=-1$. The Lie brackets are then given by $6$ structure functions 
according to the following relations: 
\begin{align*}
\lbrack X_{1},X_{3}]& =c_{13}^{1}X_{1}+c_{13}^{2}X_{2} \\
\lbrack X_{2},X_{3}]& =c_{23}^{1}X_{1}+c_{23}^{2}X_{2} \\
\lbrack X_{1},X_{2}]& =c_{12}^{1}X_{1}+c_{12}^{2}X_{2}+X_{3}.
\end{align*}

The coframe dual to the frame $\{X_1,X_2,X_3\}$ is an adapted 
coframe and denoted $\{\theta_1,\theta_2,\theta_3\}$. Using Cartan's 
formula we get following structure equations for the adapted coframe: 
\begin{align*}
d \theta_1 &= c_{12}^1 \theta_2 \wedge \theta_1 + c_{13}^1 \theta_3 
\wedge
\theta_1 + c_{23}^1 \theta_3 \wedge \theta_2 \\
d \theta_2 &= c_{12}^2 \theta_2 \wedge \theta_1 + c_{13}^2 \theta_3 
\wedge
\theta_1 + c_{23}^2 \theta_3 \wedge \theta_2 \\
d \theta_3 &= \theta_2 \wedge \theta_1
\end{align*}

Since $d^2\theta_3=0$ we immediately get that $c^1_{13}+c^2_{23}=0$, 
hence
letting $c^1_{13}=c$ and $c^2_{23}=-c$ we get 
\begin{align}
[X_1,X_3] &= c X_1 + c_{13}^2 X_2  \notag \\
[X_2,X_3] &= c_{23}^1 X_1 - c X_2  \label{lbraks} \\
[X_1,X_2] &= c_{12}^1 X_1 + c_{12}^2 X_2 + X_3.  \notag
\end{align}
and 
\begin{align}
d \theta_1 &= c_{12}^1 \theta_2 \wedge \theta_1 + c \theta_3 \wedge 
\theta_1
+ c_{23}^1 \theta_3 \wedge \theta_2  \notag \\
d \theta_2 &= c_{12}^2 \theta_2 \wedge \theta_1 + c_{13}^2 \theta_3 
\wedge
\theta_1 - c \theta_3 \wedge \theta_2 \label{SubLorStrucEqs}\\
d \theta_3 &= \theta_2 \wedge \theta_1. \notag
\end{align}

If $\mathcal{H}$ denotes the contact distribution and  $g$ denotes the  
sub-Lorentzian metric on $\mathcal{H}$, then the filtration of the 
tangent bundle is given by 
\begin{align*}
F^{0}  & = \{ 0 \}, \quad  F^{-1}= \mathcal{H}, \quad F^{-2} = 
\mathcal{H}+ [\mathcal{H}, \mathcal{H}] ,
\end{align*}
where $ \func{gr}_{-1} TM=\mathcal{H}$. The type in such cases is 
given by the Heisenberg algebra. In particular 
\begin{equation*}
\mathfrak{m}=\mathfrak{m}_{-2}\oplus \mathfrak{m}_{-1},
\end{equation*}
where $\mathfrak{m}_{-1}=\langle e_1,e_2\rangle$, 
$\mathfrak{m}_{-2}=\langle e_3\rangle$ and $[e_1,e_2]=e_3$.

An adapted to sub-Lorentzian 3-dimensional contacts structure  weighted 
frame takes the form
\begin{align*}
\Phi(p,e_1) & = Y_1 \in H \\
\Phi(p,e_2) & = Y_2\in H \\
\Phi(p,e_3) & = [Y_1,Y_2] \in TM/H \\
\end{align*}
where $g(Y_{1},Y_{1})=-1,$ $g(Y_{1},Y_{2})=0,$ $g(Y_{2},Y_{2})=1.$ 
Since our interest is in $ts$-isometric equivalence, we consider the 
$SO_{1,1}^+(\mathbb{R})$-principle bundle of ts-oriented weighted 
frames. It follows that adapted to the structure frame is of the form
\begin{align*}
E_1(p) & = a_{12}(p) X_1 + a_{22}(p) X_2, \\
E_2(p) & = a_{12}(p) X_1 + a_{22}(p) X_2, \\
E_3(p) & = b_{1}(p) X_1 + b_{2}(p) X_2 + X_3,
\end{align*}
where $\begin{psmallmatrix*}
a_{11}(p) & a_{12}(p)\\
a_{21}(p) & a_{22}(p)
\end{psmallmatrix*}\in\SO$ and $X_1,X_2,X_3$ are as in \eqref{lbraks}.


For sub-Lorentzian structures on a contact three manifold, the pair 
$(\mathfrak{m},\mathfrak{g}_0)$ consists of the Heisenberg algebra 
$\mathfrak{m} = \textrm{span}\, \{e_1,e_2,e_3=[e_1,e_2]\}$ and 
$\mathfrak{g}_0$ is spanned by $\{e_4\}$ with relations 
\begin{equation*}
[e_1,e_2]=e_3,\, [e_4,e_1]=e_2,\, [e_4,e_2]=e_1. 
\end{equation*}
\begin{lem}	\label{lem:tanaka}  
	The Tanaka prolongation for this structure is 
	$\mathfrak{g}(\mathfrak{m},\mathfrak{g}_0)=\mathfrak{m}\oplus 
	\mathfrak{g}_0$.
\end{lem}
\begin{proof} Consider an arbitrary element $\varphi\in 
	\mathfrak{g}_0\otimes\mathfrak{g}_{-1}^*\oplus\mathfrak{g}_{-1}\otimes\mathfrak{g}_{-2}^*$
	in the first prolongation. Then 
	\begin{align*}
	0&=\varphi([e_1,e_3])=[\varphi(e_1),e_3]+[e_1,\varphi(e_3)]=[e_1,\varphi(e_3)],
	 \\
	0&=\varphi([e_2,e_3])=[\varphi(e_2),e_3]+[e_2,\varphi(e_3)]=[e_2,\varphi(e_3)].
	\end{align*}
	Since $\varphi(e_3)\in \mathfrak{g}_{-1}$ it must be equal to $0$. 
	Let $\varphi(e_1)=a_1 e_4$ and $\varphi(e_2)=a_2 e_4.$ Then the 
	following
	equality shows that $\varphi=0$: 
	\begin{equation*}
	0=\varphi(e_3)=\varphi([e_1,e_2])= [a_1 e_4 ,e_2] +[e_1,a_2 e_4 ]= 
	a_1 e_1 -
	a_2 e_2. 
	\end{equation*}
\end{proof}

\subsection{Normal Cartan geometry associated with 3-dimensional 
sub-Lorentzian contact structures}

As shown in Lemma \ref{lem:tanaka}, the infinitesimal flat model for 
sub-Lorentzian structures on contact $3$--manifolds is given by the 
$4$-dimensional graded Lie algebra 
\[ \mathfrak{g}=\mathfrak{g}_{-2} \oplus \mathfrak{g}_{-1} \oplus \mathfrak{g}_{0} \] with basis $\{e_1,e_2,e_3,e_4\}$ satisfying the following relations: 
\begin{equation*}
[e_1,e_2]=e_3,\, [e_4,e_1]=e_2 \quad \textrm{and} \quad  [e_4,e_2]=e_1, 
\end{equation*}
where $\mathfrak{g}_{-2}=\mathrm{span}\, \{e_3\}$, $\mathfrak{g}_{-1} = 
\mathrm{span}\, \{e_1, e_2\}$ and $\mathfrak{g}_{0}=\mathrm{span}\, 
\{e_4\}$. 

A Cartan connection for sub-Lorentzian contact structure on a 
$3$-dimensional manifold $M$ is defined on a principle 
$SO_{1,1}^+(\R)$-bundle $\G$. Since we are interested in local 
equivalence, we can assume $\G$ is the trivial bundle 
$SO_{1,1}^+(\R)\times U$ where $U$ is an open subset of $M$. 

Consider an arbitrary Cartan connection $\tilde \omega = \sum_{i=1} 
\tilde \omega_i e_i \colon T\G \to \g$ for sub-Lorentzian structure on a $3$-dimensional contact manifold $M$. The curvature of this 
connection is 	
\begin{equation*}
\tilde \Omega=d \tilde \omega+ \frac12 [\tilde \omega,\tilde 
\omega]=\sum_{l=1}^4 \sum_{1\le i < j \le 3} k^l_{ij} e_l  \ot \tilde 
\omega_i \wedge \tilde \omega_j\colon \wedge^2 T\G \to \g, 
\end{equation*}
and the corresponding curvature function has the form
\begin{equation*}
\tilde k =\sum_{l=1}^4 \sum_{1\le i < j \le 3} k^l_{ij} e_l \ot e_i^* 
\wedge 
e_j^*\colon \G \to \Hom(\wedge^2 \g_-,\g).
\end{equation*}

\begin{prop}
	\label{prop:NormalConn} For an arbitrary sub-Lorentzian structure 
	there exists a unique Cartan connection $\tilde \omega = \sum_i 
	\tilde\omega_i e_i \colon T\G \to \g$ with the curvature function 
	taking values in the following 6-dimensional $SO_{1,1}^+(\R)$-module 
	$\tilde N$: 
	\begin{align*}
		& e_1\otimes e^*_1 \wedge e^*_3 - e_2 \otimes e^*_2\wedge 
		e^*_3;\,\, 
		e_1\otimes
		e^*_2\wedge e^*_3;\,\, e_2\otimes e^*_1\wedge e^*_3; \\
		& e_1\otimes e^*_1\wedge e^*_3+ e_2\otimes e^*_2\wedge e^*_3 
		;\,\,
		e_4\otimes e^*_1\wedge e^*_3;\,\, e_4\otimes e^*_2\wedge e^*_3.
	\end{align*}
\end{prop}

\begin{proof}
In order to use Theorem \ref{prop:morimoto} we need to show that 
$\tilde N$ is an $SO_{1,1}^+(\R)$-module and is complementary to the 
image of Lie algebra differential 
\[ \partial \colon 
\func{Hom}(\mathfrak{g}_-, \mathfrak{g})_+ \to\func{Hom}(\wedge^2 
\mathfrak{g}_-,\mathfrak{g})_+ .\]
The image of the differential $\partial$ is generated by the following $5$ elements: 
	\begin{align*}
		& e_1\otimes e^*_1 \wedge e^*_2- e_3\otimes e^*_2\wedge 
		e^*_3;\,\, 
		e_2\otimes
		e^*_1\wedge e^*_2 + e_3\otimes e^*_1\wedge e^*_3; \\
		& e_1\otimes e^*_1\wedge e^*_2 ;\,\, e_2\otimes e^*_1\wedge 
		e^*_2 ;\,\,
		e_4\otimes e^*_1\wedge e^*_2 - e_2\otimes e^*_1\wedge e^*_3 - 
		e_1\otimes
		e^*_2\wedge e^*_3,
	\end{align*}
	and doesn't intersect $\tilde N$.
	The space $\func{Hom}(\wedge^2 \mathfrak{g}_-,\mathfrak{g})_+$ is
	$11$-dimensional, therefor $\tilde N$ is complementary to $\func{im} 
	\partial$. 
	
	An 	element $h \in SO_{1,1}^+(\mathbb{R})$ acts naturally on 
	2-dimensional
	space $\langle e_1,e_2\rangle $, acts by inverse transform $h^{-1}$ on $%
	\langle e^*_1 , e^*_2 \rangle $ and trivially by identity on 
	$\{e_3,e_4, e^*_3 , e^*_4 \}$.
	This defines and action of $SO_{1,1}^+(\R)$ on $\func{Hom}(\wedge^2 
	\mathfrak{g}_-,%
	\mathfrak{g})_+$ and we can see that $\tilde N$ is in fact an 
	$SO_{1,1}^+(\R)$-module.
\end{proof}

\begin{prop}\label{effective}
		Assume that the the curvature function of the Cartan connection 
		belongs to the module $\tilde N$ defined in 
	Proposition \ref{prop:NormalConn}. Then the coefficient of 
\[ e_1\otimes e^*_1\wedge e^*_3+  e_2\otimes e^*_2\wedge e^*_3 \]
	 is equal to zero. 

	 Furthermore, the coefficients of $e_4\otimes e^*_1\wedge e^*_3$ and 
	 $e_4\otimes e^*_2\wedge e^*_3$ are linear combinations of the covariant 
	 derivatives of the coefficients of 
	 \[
	 e_1\otimes e^*_1 \wedge e^*_3- e_2\otimes e^*_2\wedge 
	 e^*_3, \quad e_1\otimes
	 e^*_2\wedge e^*_3 \quad \textrm{and} \quad e_2\otimes e^*_1\wedge 
	 e^*_3 
	 .\]
\end{prop}

\begin{proof}
 Let $\tilde\Omega = d \tilde \omega + \frac12 [\tilde{\omega},\tilde{\omega}]$ 
 be the curvature of the normal Cartan connection $\tilde{\omega}$. The fact 
 that curvature function $\tilde{k}(\cdot,\cdot)=\tilde\Omega\left(\tilde 
 \omega^{-1}(\cdot),\tilde 
 \omega^{-1}(\cdot)\right)$ belongs to $\tilde N$ is equivalent by definition to
 \begin{align}
 \tilde \Omega =& 
 k_1 (e_1 \otimes \tilde\omega_1 \we \tilde\omega_3 + e_2 \otimes \tilde\omega_2\we \tilde\omega_3) +
 k_2 (e_1 \otimes \tilde\omega_1\we \tilde\omega_3-e_2 \otimes \tilde\omega_2\we \tilde\omega_3) +
 \nonumber \\
 & k_3 e_1 \otimes \tilde\omega_2\we \tilde\omega_3 + 
 k_4 e_2 \otimes \tilde\omega_1\we \tilde\omega_3 +
 k_5 e_4 \otimes \tilde\omega_1\we \tilde\omega_3 + 
 k_6 e_4 \otimes \tilde\omega_2\we \tilde\omega_3 \label{curvtexpress} .
 \end{align}

 The Bianchi identity \cite[p. 193]{Sharpe} states that 
	\begin{equation}  \label{eq:bianchi}
		d\tilde\Omega + [\tilde\omega,\tilde\Omega]=0.
	\end{equation}
        By \eqref{curvtexpress},   $d \tilde \Omega$ has a trivial 
        projection onto $e_3$. Moreover, using $\tilde \omega = \sum_i 
        e_i \otimes \tilde\omega_i $ and directly calculating gives 
        \[[\tilde\omega,\tilde\Omega]= 2 k_1 e_3 \otimes  
        \tilde\omega_1 \we \tilde\omega_2 \we \tilde\omega_3 	
        \mod{\langle e_1,e_2 \rangle} \] and so $k_1=0$.

Consider now the projection of Bianchi identity on $\g_{-1}=\langle 
e_1,e_2\rangle.$ The $\g_{-1}$-part of $[\tilde\omega,\tilde\Omega]$ is given by
\begin{equation}\label{p3:eq1}
[e_1\tilde\omega_1 + e_2\tilde\omega_2, k_5 e_4 \otimes 
\tilde\omega_1\we \tilde\omega_3 + 
k_6 e_4 \otimes \tilde\omega_2\we \tilde\omega_3 ] = k_5 e_1  
\tilde\omega_1\we \tilde\omega_2\we \tilde\omega_3 - k_6 e_2  
\tilde\omega_1\we \tilde\omega_2\we \tilde\omega_3.
\end{equation}
 The formula for $\g_{-1}$-part of $d\tilde\Omega$ is also 
 straightforward:
 \begin{equation} \label{p3:eq2}
 d\tilde\Omega = \left( (\tilde X_1 k_3 - \tilde 
 X_2 k_2) e_1 -  
 (\tilde X_2 k_4 + \tilde X_1 k_2) e_2 \right) \tilde\omega_1\we 
 \tilde\omega_2\we \tilde\omega_3 \mod{\tilde\omega_4},
 \end{equation}
 where $\tilde X_i=\tilde\omega^{-1}(e_i)\in \Gamma(T \G)$ are universal
 covariant differentiations defined by the normal Cartan connection, 
 see \cite[p. 194]{Sharpe}.
 Comparing \eqref{p3:eq1} and \eqref{p3:eq2} we conclude that
 \[ k_5=\tilde  X_2 k_2 - \tilde X_1 k_3 ,\,\, k_6=\tilde X_2 k_4 + 
 \tilde X_1 k_2. \]
\end{proof}

We denote by $N$ the ``essential'' part of $\tilde N$ which by 
Proposition \ref{effective} is the submodule generated by 
\begin{equation} \label{modN}
\begin{aligned}
K &= e_1\otimes  e^*_1 \wedge  e^*_3- e_2\otimes e^*_2\wedge e^*_3, 
\\
X &= e_1\otimes e^*_2\wedge e^*_3, \\
Y &= e_2\otimes e^*_1\wedge e^*_3,
\end{aligned}
\end{equation}


Since $SO_{1,1}^+(\R)$ acts on $e^*_3$ by identity we see that actually 
$SO_{1,1}^+(\R)$ acts on $K$, $X$ and $Y$ as on $e_1\otimes 
e^*_1-e_2\otimes e^*_2$, $e_1\otimes e^*_2$ and 
$e_2\otimes e^*_1$. The later is exactly the adjoint action of 
$SO_{1,1}^+(\mathbb{R})$ on $\gsl_2(\mathbb{R})$ given by $A 
\to T A T^{-1}$ where 
\begin{equation}
K=\begin{psmallmatrix}
1 & 0 \\
0 & -1
\end{psmallmatrix}, \quad  
X=
\begin{psmallmatrix}
0 & 1 \\
0 & 0
\end{psmallmatrix},  \quad  
Y=
\begin{psmallmatrix}
0 & 0 \\
1 & 0
\end{psmallmatrix}   \quad  \textrm{and} \quad 
T=
\begin{psmallmatrix}
\cosh(t) & \sinh(t) \\
\sinh(t) & \cosh(t)
\end{psmallmatrix}. \label{SL2onN}
\end{equation}


\subsection{Local computation of normal Cartan connection}

Consider a principle $SO_{1,1}^+(\R)$-bundle $\G=SO_{1,1}^+(\R)\times 
U$ where $U$ is an open subset of $M$. Let 
$\tilde{\omega}\colon 
T\G\to \g$ be an arbitrary Cartan connection adapted to the sub-Lorentzian 
structure $(M,\mathcal{H},g)$. For an arbitrary section $s\colon U\to 
\G$ the 
corresponding Cartan gauge $\omega=s^*\tilde\omega\colon 
TU\to\mathfrak{g}$ 
adapted to the sub-Lorentzian structure has the form: 
\begin{equation*}
\omega=e_1\otimes\omega_1+e_2\otimes\omega_2+e_3\otimes\omega_3+e_4\otimes%
\omega_4,
\end{equation*}
where 
\begin{align*}
\omega_1&=\bar a_{11}\theta_1+\bar a_{12}\theta_2 + \bar 
\alpha_1\theta_3, \\
\omega_2&=\bar a_{21}\theta_1+ \bar a_{22}\theta_2+\bar 
\alpha_2\theta_3, \\
\omega_3&=\theta_3, \\
\omega_4&=\bar \beta_1\theta_1+\bar \beta_2\theta_2+\bar \beta_3\theta_3
\end{align*}
and the matrix $(\bar a_{ij})$ is an element of $SO_{1,1}^+(\mathbb{R})$. 
Changing the section by the right action 
of suitable element in $SO_{1,1}^+(\mathbb{R})$ according to formula 
\ref{intro:cartan_change} gives the following lemma.

\begin{lem}
\label{lem:section}  There exists a unique section such that a Cartan gauge
adapted to  the sub-Lorentzian structure has the form: 
\begin{align*}
\omega_1&=\theta_1 + \alpha_1\theta_3, \\
\omega_2&=\theta_2 + \alpha_2\theta_3, \\
\omega_3&=\theta_3, \\
\omega_4&=\beta_1\theta_1+\beta_2\theta_2+\beta_3\theta_3.
\end{align*}
\end{lem}

If $\omega=s^* \tilde \omega$ then $ \Omega= s^* \tilde \Omega \colon\wedge^2 
TM\to \mathfrak{g}$ is given by $ \Omega=d\omega+\frac12[\omega,\omega]$ which 
for our specific case takes the
form: 
\begin{align*}
\Omega&=e_1\otimes\Omega_1+e_2\otimes\Omega_2+e_3\otimes\Omega_3+e_4\otimes \Omega_4,
\end{align*}
where 
\begin{align*}
\Omega_1&=d\omega_1-\omega_2\wedge\omega_4, \\
\Omega_2&=d\omega_2-\omega_1\wedge\omega_4, \\
\Omega_3&=d\omega_3-\omega_2\wedge\omega_1, \\
\Omega_4&=d\omega_4.
\end{align*}

In accordance with Propositions \ref{prop:NormalConn}  and \ref{effective}, the normal Cartan connection satisfies the conditions 
\begin{align*}
\Omega_3 &=0, \\
\Omega_i &=0\mod{\omega_3},\,i=1,2,4.
\end{align*}
The first condition is : 
\begin{equation*}
\Omega_3= d\omega_3-\omega_2\wedge\omega_1 =
\alpha_1\omega_3\wedge\omega_2-\alpha_2\omega_3 \wedge \omega_1=0.
\end{equation*}
Therefore, $\alpha_1=\alpha_2=0$. The second normalization condition is $%
\Omega_1 \mod \om_3=\Omega_2 \mod \om_3=0$. This condition gives us: 
\begin{align*}
\Omega_1&=d\omega_1-\omega_2\wedge\omega_4= (\beta_1-c^1_{12}) \omega_1
\wedge \omega_2 =0\mod \om_3, \\
\Omega_2&=d\omega_2-\omega_1\wedge\omega_4= (-\beta_2-c^2_{12}) \omega_1
\wedge \omega_2 =0\mod \om_3,
\end{align*}
From the formulas above we obtain $\beta_1= c^1_{12}$ and $\beta_2=-c^2_{12}.
$ The last normalization condition is $\Omega_4 \mod \om_3=0:$ 
\begin{equation*}
\Omega_4=d\omega_4=( -\beta_3 - X_2(\beta_1) + X_1(\beta_2) -\beta_1 
c_{12}^1 -
\beta_2 c_{12}^2 ) \omega_1 \wedge \omega_2 =0\mod \om_3
\end{equation*}
We obtain that $\beta_3 = (c_{12}^2)^2 - (c^1_{12})^2 -  X_1(c^2_{12}) 
- X_2(c^1_{12}).$

To summarize, the coefficients of $\Omega$ have the form: 
\begin{align} \label{curvf}
\Omega_1=& -c \omega_1\wedge\omega_3 -
(c^1_{23}+\beta_3)\omega_2\wedge\omega_3 , \\
\Omega_2=& -(c^2_{13}+\beta_3)\omega_1\wedge\omega_3 + c
\omega_2\wedge\omega_3, \\
\Omega_3=&0, \\
\Omega_4=&\left(X_1(\beta_3)-X_3(\beta_1) - \beta_1 c - \beta_2 
{c^2_{13}}
\right) \omega_1\wedge\omega_3 \\
& \quad + \left(X_2(\beta_3) - X_3( \beta_2) -\beta_1 c_{23}^1 + 
\beta_2 c
\right) \omega_2\wedge\omega_3.
\end{align}
Where $\beta_1= c^1_{12}$, $\beta_2=-c^2_{12}$ and $\beta_3 = 
(c_{12}^2)^2 - (c^1_{12})^2 -  X_1(c^2_{12}) 
- X_2(c^1_{12}).$


\subsection{Invariants of sub-Lorentzian structure} \label{sec:Inv}A 
normal Cartan connection is a special type of absolute parallelism. 
The problem of equivalence of absolute parallelisms is a classical 
subject and was studied in details for example in \cite{Sternberg}. 
In particular, local invariants of an absolute parallelism are precisely its structure function and its consecutive 
covariant derivatives. A finite number of these invariants uniquely (up 
to local equivalence) determines the absolute parallelism.

Applied to a normal Cartan connection, this means that all local 
invariants of a given Cartan connection can be derived from its structure 
function and consecutive covariant derivatives. However, due to 
 $H$-equivariance of the structure function $\tilde 
k\colon\G\to\Hom(\we^2\g_-,\g)$, one can simplify $k$ by introducing a 
canonical section $s\colon M\to \G$ and considering a canonical 
pullback $k=s^* \tilde k$. This allows us to obtain invariants 
generated by $k$ that are defined on the manifold $M$ instead of the 
principal bundle $\G$.

Let $k_N$ be the part of the curvature function taking values in the module 
$N$ generated by \eqref{modN}. According to Proposition 
\ref{effective}, the entire curvature function $k$ can be expressed 
through $k_N$ using covariant differentiation. Therefore we need only 
 focus our attention on $k_N\colon M\to \gsl_2(\R)$. 

Under the adjoint action,  the $SO_{1,1}^+(\R)$-module 
$\gsl_2(\mathbb{R})$ has 2 irreducible submodules. One is generated by  
the matrix 
\begin{equation}  \label{e01}
f_0=\begin{pmatrix}
0 & -1 \\ 
-1 & 0
\end{pmatrix}.
\end{equation}
and the other generated by the pair of matrices 

\begin{equation}  \label{e02d}
f_1=\begin{pmatrix}
-1 & 0 \\ 
0 & 1
\end{pmatrix} \quad \textrm{and} \quad 
f_2=\begin{pmatrix}
0 & -1 \\ 
1 & 0
\end{pmatrix}.
\end{equation}

If we write $k_N=\kappa f_0 + a f_1 +bf_2$ then
\begin{align*}
\kappa &= \beta_3 + \frac{c^1_{23}+c^2_{13}}2= (c_{12}^2)^2 - (c^1_{12})^2 -
X_1(c^2_{12}) - X_2(c^1_{12})+ \frac{c^1_{23}+c^2_{13}}2,\\
a&=c\\
b&=(c_{23}^1-c_{13}^2)/2
\end{align*}
Under the change of section $\bar s=R_h(s),$ $h\colon 
M\to SO^+_{1,1}$ we get \[ \bar k_N=\bar s^*\tilde k_N = Ad_{h^{-1}} 
(s^*\tilde k_N) = Ad_{h^{-1}}(k_N). \]
 Since 
$SO_{1,1}^+(\mathbb{R})$ acts on \eqref{e01} as the identity, 
 $\kappa$ doesn't depend on the choice of section 
and is an invariant. We summarize this observation in the following 
proposition. 

\begin{prop}
The following expression  
\begin{equation}
\kappa = \beta_3 + \frac{c^1_{23}+c^2_{13}}2= (c_{12}^2)^2 - (c^1_{12})^2 -
X_1(c^2_{12}) - X_2(c^1_{12})+ \frac{c^1_{23}+c^2_{13}}2.
\end{equation}
is an invariant of time-space orientation preserving structure. 
\end{prop}

Consider now the submodule of $\gsl_2(\R)$ generated by
\begin{equation}  \label{e02}
\begin{pmatrix}
a & b \\ 
-b & -a%
\end{pmatrix}%
.
\end{equation}
The corresponding part of $k_N$ is 
\begin{equation}  \label{inv_h}
 h=af_1+bf_2= 
\begin{pmatrix}
c & \frac{c^1_{23}-c^2_{13}}2 \\ 
\frac{c^2_{13}-c^1_{23}}2 & -c%
\end{pmatrix}%
\end{equation}
and it depends on the choice of section $s\colon M\to \G$.
To obtain an absolute invariant of the structure we factor the expression at \eqref{inv_h} by the action of $SO_{1,1}^+(\R)$.

\begin{prop}
\label{p2} For every 3-dimensional contact sub-Lorentzian $ts$-oriented manifold
there exists a section $s$ such that the invariant $ h$ has the 
following form at a given point:

\begin{enumerate}
\item \label{p2:1} If $ h=0$ then:
\begin{equation*}
 h\in \left\{
\begin{pmatrix}
0 & 0 \\ 
0 & 0
\end{pmatrix}
,\,
\begin{pmatrix}
1 & \pm 1 \\ 
\mp 1 & -1
\end{pmatrix}
,\, 
\begin{pmatrix}
-1 & \pm 1 \\ 
\mp 1 & 1
\end{pmatrix}
\right\}. 
\end{equation*}

\item \label{p2:2}  If $\det h>0$ then:
\begin{equation*}
 h=
\begin{pmatrix}
0 & \chi \\ 
-\chi & 0
\end{pmatrix}.
\end{equation*}

\item  \label{p2:3}  If $\det  h<0$ then:
\begin{equation*}
 h=
\begin{pmatrix}
\chi & 0 \\ 
0 & -\chi
\end{pmatrix}.
\end{equation*}
\end{enumerate}

Moreover if $ h\neq 0$ such a section is unique.
\end{prop}

\begin{proof} If  $T=
\begin{psmallmatrix}
\cosh(t) & \sinh(t) \\
\sinh(t) & \cosh(t)
 \end{psmallmatrix} \in SO_{1,1}^+(\R)$ then the adjoint action on $ h$ 
 is given by
\begin{equation}
T^{-1} h\, T = \begin{pmatrix}
\cosh(2t) a+\sinh(2t) b & \sinh(2t) a + \cosh(2t) b \\ 
-\sinh(2t) a -\cosh(2t) b & -\cosh(2t) a-\sinh(2t) b \end{pmatrix}. 
\label{adjTh}
\end{equation}
If  $ \det  h = b^2-a^2> 0$ then the equation 
\begin{equation*}
\cosh(2t) a+\sinh(2t) b = 0
\end{equation*}
has the solution  $t=\frac14\ln(\frac{b-a}{b+a})$ and $h$ takes the 
form in item (\ref{p2:2}) with
\begin{equation}  \label{p2:eq1}
\chi = \sgn(a+b)\sqrt{b^2-a^2}.
\end{equation}

Similarly, if   $ \det  h < 0$  then the equation 
\begin{equation*}
\sinh(2t) a + \cosh(2t) b = 0
\end{equation*}
has the solution  $t=\frac14\ln(\frac{a-b}{b+a})$ and $h$ takes the 
form in item (\ref{p2:3}) with
\begin{equation}  \label{p2:eq2}
\chi = \sgn(a+b)\sqrt{a^2-b^2}.
\end{equation}
If $\det h=b^2-a^2=0
$ then
\begin{equation}  \label{e03}
h=
\begin{pmatrix}
a & \pm a \\ 
\mp a & -a
\end{pmatrix}
\end{equation}
and the adjoint action of $T$ is simply $T^{-1} h\, T = \exp(\pm 2t) 
h$,  which depending on $\sgn(a)$, has exactly one of the forms in item 
(\ref{p2:1}).
\end{proof}

Since $a=c$ and $b=(c_{23}^1-c_{13}^2)/2$ we have the following Corollary.

\begin{cor}
Assume that $\det h\neq 0$ for sub-Lorentzian contact structure. Then 
the following expression is a local invariant of $ts$-oriented 
structures: 
\begin{equation*}
\chi =\func{sgn}\left(c+\frac{c^1_{23}-c^2_{13}}2\right)\sqrt{\left| \left(\frac{c^1_{23}-c^2_{13}}2\right)^2-c^2 \right|}.
\end{equation*}
\end{cor}

\begin{cor}
\label{cor2} If a sub-Lorentzian contact structure on a contact $3$ 
manifold $M$ satisfies $h\neq 0$, then Proposition \ref{p2} defines a 
unique normal frame $\theta=s^*\omega_-$.
\end{cor}
\begin{proof}
The change of section $s\to s\cdot h $, $h\colon M\to \SO$ for the 
Cartan connection is equivalent to the change $\theta \to 
\Ad_h(\theta)$ of associated frame 
$\theta=s^*\omega_-$ 
 due to the formula \eqref{intro:cartan_change}. Therefor existence and 
 uniqueness follows from Proposition \ref{p2}.
\end{proof}

\section{Classification of left-invariant sub-Lorentzian structures}
\subsection{Classification of real $3$-dimensional Lie algebras}

To begin we review the classification of real $3$-dimensional Lie algebras
following \v{S}nobl and Winternitz \cite{SnobWint}. If $\mathfrak{g}$ is a real $3$%
-dimensional Lie algebra then we define $\mathfrak{g}^{(1)}=[\mathfrak{g},%
\mathfrak{g}]$ and divide into the cases : $\mathrm{dim}\, \mathfrak{g}%
^{(1)} =0,1,2,3$. The cases $\mathrm{dim}\, \mathfrak{g}^{(1)} =0$ is 
the abelian algebra and the case $\mathrm{dim}\, \mathfrak{g}^{(1)} =1$
determines two classes, namely the Heisenberg algebra and the affine
algebra. The cases $\mathrm{dim}\, \mathfrak{g}^{(1)} =2$ give rise to
significantly more classes corresponding to eigenvalues of $\func{ad}_X 
|_{\mathfrak{g}^{(1)}}$, where $X\notin \mathfrak{g}^{(1)}$. Finally, 
in the case $\mathfrak{g}^{(1)} =3$  there are only $2$ 
non-isomorphic semi-simple Lie algebras, namely $\gsl_2(\R)$ and 
$\gsu_2(\R)$ 

\begin{thm}
\label{3class} Any $3$-dimensional Lie algebra is isomorphic to exactly one
of the following algebras:
\begin{align*}
\mathrm{dim}\, \mathfrak{g}^{(1)} &=0 \\
L(3,0):\, &\, [E_1,E_2]=0, \, \, [E_1,E_3]=0, \, \, [E_2,E_3]=0,\, \,  (\R^3) \\
\\
\mathrm{dim}\, \mathfrak{g}^{(1)} &=1 \\
L(3,1):\, &\, [E_1,E_2]=E_3, \, \, [E_1,E_3]=0, \, \, [E_2,E_3]=0, \, 
\, (\textrm{Heisenberg})\\
L(3,-1):\, &\, [E_1,E_2]=E_1, \, \, [E_1,E_3]=0, \, \, [E_2,E_3]=0,\, \, (A^+(\mathbb{R}) \oplus \mathbb{R}) \\
\\
\mathrm{dim}\, \mathfrak{g}^{(1)} &=2 \\
L(3,2,\eta):\, &\, [E_1,E_2]=0, \, \, [E_1,E_3]=E_1, \, \, [E_2,E_3]=\eta
E_2, \, \, 0<|\eta| \leq 1, \\
L(3,4,\eta):\, & \, [E_1,E_2]=0, \, \, [E_1,E_3]=\eta E_1- E_2, \, \,
[E_2,E_3]= E_1 + \eta E_2, \, \, \eta \in [0, \infty), \\
L(3,3):\, &\, [E_1,E_2]=0 , \, \, [E_1,E_3]=E_1, \, \, [E_2,E_3]= E_1 + E_2,\\
\\
\mathrm{dim}\, \mathfrak{g}^{(1)} &=3 \\
L(3,5):\, &\, [E_1,E_2]=E_1, \, \, [E_1,E_3]=-2E_2, \, \, 
[E_2,E_3]=E_3, \, \, \gsl_2(\R), \\
L(3,6):\, &\, [E_1,E_2]=E_3, \, \, [E_1,E_3]=-E_2, \, \, [E_2,E_3]= 
E_1, \, \, \gsu_2(\R).
\end{align*}
\end{thm}

We review the proof of the classification theorem as it provides the
procedures for putting a given algebra into its canonical form.

\begin{proof}
If $\mathrm{dim}\, \mathfrak{g}^{(1)} =0$ then $\mathfrak{g}$ is the three
dimensional abelian Lie algebra $L(3,0)$. If $\mathrm{dim}\, \mathfrak{g}%
^{(1)} =1$ then there exists $Z \in \mathfrak{g}$ such that $\mathfrak{g}%
^{(1)} = \mathrm{span}\, \{Z\}$ and so it follows that for any basis of the
form $\{X,Y,Z\}$ we have 
\begin{equation*}
[X,Y] = \alpha_1 Z, \quad [X,Z]= \alpha_2 Z, \quad [Y,Z]=\alpha_3 Z. 
\end{equation*}

Direct calculation shows that $[\mathfrak{g},[\mathfrak{g},\mathfrak{g}]]=\{0\}$ if and only if $\alpha_2=\alpha_3=0$. Moreover, if $\alpha_2=\alpha_3=0$ then $\mathfrak{g}$ is the Heisenberg algebra. Indeed,
if $E_1=X$, $E_2=Y$ and $E_3=\alpha_1Z$ then 
\begin{equation*}
[E_1,E_2]=E_3, \quad [E_1,E_3]=0, \quad [E_2,E_3]=0.
\end{equation*}

If $\alpha_3 \ne 0$ then 
\begin{equation*}
E_1=Z, \quad E_2=-\frac{1}{\alpha_3}Y, \quad E_3= \alpha_3 X - \alpha_2 Y +
\alpha_1 Z 
\end{equation*}
is a basis with bracket relations: 
\begin{align}
[E_1,E_2]=E_1, \quad [E_1,E_3]=0, \quad [E_2,E_3]=0.  \label{L3-1}
\end{align}
Similarly, if $\alpha_2 \ne 0$ then 
\begin{equation*}
E_1=Z, \quad E_2=-\frac{1}{\alpha_2}X, \quad E_3= -\alpha_3 X + \alpha_2 Y -
\alpha_1 Z 
\end{equation*}
is a basis with the same bracket relations as in \eqref{L3-1} above.

The Lie algebra defined by \eqref{L3-1} is denoted $L(3,-1)$ and has the
decomposition $L(3,-1)=L(2,1) \oplus L(1,0)$ where $L(2,1)$ is the
subalgebra spanned by $\{E_1,E_2\}$ and $L(1,0)=\mathrm{span}\, \{E_3\}$.

Next we assume $\mathrm{dim}\, \mathfrak{g}^{(1)} =2$. There are only 
two $2$-dimensional Lie algebras. Therefor $\g^{(1)}$ is abelian since 
$\g^{(1)}$ doesn't contain semi-simple elements ($\g$ is solvable).
%
Furthermore,  $\mathrm{rank}\, \func{ad}_X|_{\mathfrak{g}^{(1)}}=2$ for 
any
nonzero $X \not\in \mathfrak{g}^{(1)}$. We conclude that the map $\func{ad}%
_X|_{\mathfrak{g}^{(1)}}$ is an isomorphism $\mathfrak{g}^{(1)} \to 
\mathfrak{g}^{(1)}$ which doesn't depend on $X\notin\mathfrak{g}^{(1)}$.


There are three subcases to consider: $\func{ad}_X | _{\mathfrak{g}^{(1)}}$
diagonalises in $GL(\mathfrak{g}^{(1)})$, $\func{ad}_X | _{\mathfrak{g}^{(1)}}$ diagonalises in $GL( {\mathfrak{g}^{(1)}} \otimes_{\mathbb{R}} 
\mathbb{C})$, $\func{ad}_X | _{\mathfrak{g}^{(1)}}$ does not diagonalises. We
get the following classes: 
\begin{align*}
L(3,2,\eta):\, &\, [E_1,E_2]=0, \, \, [E_1,E_3]=E_1, \, \, [E_2,E_3]=\eta
E_2, \, \, 0<|\eta| \leq 1, \\
L(3,4,\eta):\, & \, [E_1,E_2]=0, \, \, [E_1,E_3]=\eta E_1- E_2, \, \,
[E_2,E_3]= E_1 + \eta E_2, \, \, \eta \in [0, \infty), \\
L(3,3):\, &\, [E_1,E_2]=0 , \, \, [E_1,E_3]=E_1, \, \, [E_2,E_3]= E_1 + E_2.
\end{align*}

Case: $\func{ad}_X$ diagonalises in $GL(\mathfrak{g}^{(1)})$. Suppose 
the eigen values of $%
\func{ad}_X | _{\mathfrak{g}^{(1)}} $ are $\lambda_1, \lambda_2 \in \mathbb{R%
} $, $|\lambda_1|\geq|\lambda_2|$ with $V_1,V_2 \in \mathfrak{g}^{(1)}$ 
denoting corresponding linearly
independent eigen vectors. Then the basis 
\begin{equation*}
E_1=V_1, \quad E_2=V_2, \quad E_3=-\frac{1}{\lambda_1} X
\end{equation*}
satisfies the $L(3,2,\eta)$ bracket relations with $\eta=\frac{\lambda_2}{%
\lambda_1}$, $0< |\eta| \leq 1$.

Case:  $\func{ad}_X$ diagonalises in $GL( {\mathfrak{g}^{(1)}} 
\otimes_{\mathbb{R}} \mathbb{%
C})$. In this case the eigenvalues are a conjugate pair $(\lambda, \bar
\lambda)$ and there exists a nonzero $W=U+iV \in {\mathfrak{g}^{(1)}}
\otimes_{\mathbb{R}} \mathbb{C}$ such that $[X,W]=\lambda W$. If $\Re
\lambda/ \Im \lambda \geq 0$ then the basis 
\begin{equation*}
E_1=U, \quad E_2= V, \quad E_3=-\frac{1}{\Im \lambda} X
\end{equation*}
satisfies the $L(3,4,\eta)$ bracket relations with $\eta =\Re \lambda/ \Im
\lambda$. If $\Re \lambda/ \Im \lambda <0$ then  
\begin{equation*}
E_1=U, \quad E_2= -V, \quad E_3=\frac{1}{\Im \lambda} X
\end{equation*}
satisfies the $L(3,4,\eta)$ bracket relations with $\eta =-\Re \lambda/ \Im
\lambda$. It is known that $L(3,4,s) \simeq L(3,4,t)$ if and only if $t= \pm
s$.

Case: $\func{ad}_X$ does not diagonalises. In this we consider the 
Jordan form of $\func{ad}%
_X | _{\mathfrak{g}^{(1)}}$, in particular we can choose the basis $\{Y,Z\}$
so that $\mathrm{ad}_X$ is given by the matrix 
\begin{equation*}
\begin{pmatrix}
\lambda & 1 \\ 
0 & \lambda
\end{pmatrix}
, \quad \lambda= \frac{1}{2} \mathrm{tr}\, \mathrm{ad}_X. 
\end{equation*}
Then the basis 
\begin{equation*}
E_1=\frac{1}{\lambda}Y, \quad E_2=Z, \quad E_3=-\frac{1}{\lambda} X 
\end{equation*}
satisfies the $L(3,3)$ bracket relations.

In the case $\mathrm{dim}\, \mathfrak{g}^{(1)}=3$ we use the fact that 
there are just 2 semi-simple real Lie algebras of dimension 3. 
One can distinguish $L(3,5)$ and $L(3,6)$ via the Killing form. Indeed 
$L(3,5)=\gsl_2(\R)$ is a split real form of $\gsl_2(\C)$ with 
sign-indefinite Killing form and $L(3,6)=\gsu_2(\R)$ is a compact real 
form with negative negative definite Killing form.

\end{proof}
\subsection{Left-invariant sub-Lorentzian contact structures in 
dimension 3}
Now we are going to proof Theorem \ref{mainThm}. If the sub-Lorentzian 
structure is a left-invariant structure on a Lie 
group
then all the invariants are constant. In particular for the 
3-dimensional sub-Lorentzian case 
\begin{equation*}
\kappa = -(c^1_{12})^2+(c^2_{12})^2+ \frac{c^1_{23}+c^2_{13}}2
\end{equation*}
and  
\begin{equation*}
h= 
\begin{pmatrix}
c & \frac{c^1_{23}-c^2_{13}}2 \\ 
\frac{c^2_{13}-c^1_{23}}2 & -c
\end{pmatrix}.
\end{equation*}
In order to obtain classification we consider canonical frames given by 
Proposition \ref{p2} and Corollary \ref{cor2}.

\subsubsection{\textbf{Case} $h =0$} This case needs special 
consideration 
since 
Proposition \ref{p2} and Corollary \ref{cor2} doesn't provide canonical 
frame for this particular case.

We have $c=0$ and $ 
c_{13}^2 = c_{23}^1 = \gamma $, hence the Lie brackets are
\begin{align*}
[X_1,X_3] &= \gamma X_2 \\
[X_2,X_3] &= \gamma X_1  \\
[X_1,X_2] &= c_{12}^1 X_1 + c_{12}^2 X_2 + X_3
\end{align*}
and 
\begin{align*}
[X_1,[X_2,X_3]] +[[X_2,[X_3,X_1]]  + & [ X_3  ,[X_1,X_2]] =[X_3,[X_1,X_2]]\\
&= -\gamma c_{12}^1 X_2 - \gamma c_{12}^2 X_1 .
\end{align*}
The Jacobi identity forces $\gamma=0$ or $c_{12}^1=c_{12}^2=0$, which 
lead to the following two possible Lie algebra structures:
\begin{align*}
A: \quad [X_1,X_3] &= 0 & B: \quad [X_1,X_3] &= \gamma X_2  \\
 [X_2,X_3] &= 0 \quad   (\kappa=(c_{12}^2)^2-(c_{12}^1)^2) & [X_2,X_3] 
 &= \gamma X_1 \quad (\kappa=\gamma) \\
[X_1,X_2] &= c_{12}^1 X_1 + c_{12}^2 X_2 + X_3 & [X_1,X_2] &= X_3. 
\end{align*} 

If $c_{12}^1 =c_{12}^2 = \gamma = 0$, then both algebras $A$ and $B$ 
are isomorphic to the Heisenberg algebra $L(3,1)$. If  $\gamma \ne 0$, 
then $\gamma=\kappa$ and $B$ is isomorphic to $\gsl_2(\R)=L(3,5)$ via 
the isomorphism given by the following change of basis:
\begin{align*}
E_1 &= X_1 + X_2, \quad E_2 = \frac{1}{\kappa} X_3 \quad E_3 = 
\frac{1}{\kappa} (X_1-X_2).
\end{align*}	
Furthermore, the sub-Lorentzian metric is $-(1/ 2\kappa)K$ where $K$ is the Killing form. 
	
If $c_{12}^1\ne 0$ or $c_{12}^2 \ne 0$, then the algebra $A$ is isomorphic to $L(3,-1)$. Indeed, according to  $c_{12}^1\ne 0$ or $c_{12}^2 \ne0$, the isomorphism is given by the corresponding change of basis:
\begin{align}
1). \quad E_1 &= X_1 +\frac{c_{12}^2}{c_{12}^1} X_2  
+\frac{1}{c_{12}^1}X_3 & 2). \quad E_1 &= \frac{c_{12}^1}{c_{12}^2} X_1 
+  X_2 + \frac{1}{c_{12}^2}X_3 \nonumber\\
E_2 &= \frac{1}{c_{12}^1} X_2 & E_2 &= -\frac{1}{c_{12}^2} X_1 \label{possibles}\\
E_3 &= X_3 & E_3 &= -X_3. \nonumber
\end{align}

\begin{thm}\label{isothm1}
If $\kappa$ is nonzero and identical for $A$ and $B$, 
then the sub-Lorentzian structures are locally $ts$-isometric.
\end{thm}	
\begin{proof}

 To show this we exploit the fact that corresponding structures are constant curvature structures.

\begin{defi}
	We say that sub-Lorentzian structure is a constant curvature 
	structure if the curvature of the corresponding Cartan connection is 
	constant on the whole principle bundle. 
\end{defi}

One can see that sub-Lorentzian structure is a constant curvature 
structure only if $\SO$ acts trivially (identically) on the curvature 
function. 
Otherwise curvature function wouldn't be constant along fibers of the 
principle bundle.

Any two constant curvature structures with the same curvature are 
isomorphic due to the following theorem.
 
\begin{thm}\label{olver1418} {\cite [Thm 14.18, p. 433] {Olver1}} 
Let $\theta$ and $\bar \theta$ be two coframes on $m$-dimensional 
manifolds $M$ and $\bar M$, having the same constant structure 
functions.  Then for any pair of points $p \in M$ and $q \in \bar M$, 
there exists a unique local diffeomorphism $ \Phi: M \to \bar M$ such 
that $q=\Phi(p)$ and $\phi^* \bar \theta_i =\theta_i$ for $i = l,...,m$.
\end{thm}

If follows that there exists a local diffeomorphism of the corresponding principle bundles which preserves fibers and maps one Cartan connection to the other. Since the action of $G_0$ is preserved, the projection of this diffeomorphism gives rise to the required isometry between the underlying manifolds. 

Indeed, let $\pi_1\colon \mathcal{G}_1\to M_1$ and $\pi_2\colon \mathcal{G}_2\to M_2$ be principle bundles corresponding to sub-Lorentzian manifolds $M_1$ and $M_2$. Let $\tilde \om_1\colon T \mathcal{G}_1\to \g$ and $\tilde \om_2 \colon T \mathcal{G}_1\to \g$ be Cartan connections induced by the sub-Lorentzian structure, and let $\varphi \colon \mathcal{G}_1 \to  \mathcal{G}_2$ be the local diffeomorphism such that $\tilde \om_1=\varphi^* \tilde \om_2$. Then for any  section $s_1\colon M_1 \to  \mathcal{G}_1$ with image contained in the domain of $\varphi$, the diffeomorphism $\pi_2 \circ \varphi \circ s_1$ maps the frame adapted to sub-Lorentzian structure on $M_1$ to the frame adapted to sub-Lorentzian structure on $M_2$. Therefore this diffeomorphism automatically preserves the sub-Lorentzian structure.	

The last step in the proof is to check that A and B are constant 
curvature structures with the same curvature function. Indeed, formulas 
on page \pageref{curvf} shows that curvature function for both cases is
\[ k=\kappa(e_1\ot e^*_2 \we e^*_3 + e_2\ot e^*_1 \we e^*_3 ) \]
and $\SO$ acts trivially on it.

\end{proof}

\subsubsection{\textbf{Case} $\det h =0,$ $ h \neq 0$} In this case we 
have 
$c \in 
\{-1,1\}$ and $ c_{23}^1 - c_{13}^2 = \pm 2$, and the brackets are 
\begin{align*}
[X_1,X_3] &= c X_1 + c_{13}^2 X_2 \\
[X_2,X_3] &= (c_{13}^2  \pm 2)X_1 - c X_2 \\
[X_1,X_2] &= c_{12}^1 X_1 + c_{12}^2 X_2 + X_3.
\end{align*} 
The cases $c=1$ and $c=-1$ can be obtained one from the other by a reversal of the time orientation or a reversal of the space orientation (but not both simultaneously). The underlying transformation is isometric but not $ts$-isometric. However the underlying group is unaffected and thus we only consider the case $c=1$. 


The Jacobi identity 
\begin{align*}
[X_1,[X_2,X_3]]+[[X_2,[X_3,X_1]]+& [X_3,[X_1,X_2]] =[X_3,[X_1,X_2]]\\
&= -( c_{12}^1 + c_{12}^2  (c_{13}^2  \pm 2) ) X_1   + (c_{12}^2   
-c_{12}^1 c_{13}^2) X_2 . 
\end{align*}
implies that we must also have the following 
equations:
\[ c_{12}^1(1+ c_{13}^2 (c_{13}^2  \pm 2)) =0,  \quad \textrm{and } 
\quad   c_{12}^2=c_{12}^1 c_{13}^2. \]
There are three possible solutions:
\begin{enumerate}
\item $\pm2=-2$, $c_{13}^2 =1$, $c_{12}^2=c_{12}^1$, $\kappa=0$,
\item $\pm2= 2$, $c_{13}^2 =-1$, $ c_{12}^2=-c_{12}^1$, $\kappa=0$,
\item $c_{12}^1=0$,  $c_{12}^2=0$, $\kappa=(c_{23}^1+c_{13}^2)/2$. 
\end{enumerate}
We see that solutions (1) and (2) give rise to two families of 
sub-Lorentzian structures which couldn't be distinguished by invariants 
$\kappa$ and $h$. Therefor we introduce $\tau=c^1_{12}$ which is an 
additional invariant for these particular cases. One could check that 
for solutions (1) and (2) $\tau$ is a covariant derivative of $h$ along 
$X_1$.

In solution (1) the brackets are
\begin{align}
[X_1,X_3] &= X_1 +  X_2  \nonumber\\
[X_2,X_3] &= - (X_1 +  X_2) \label{alg1}\\
[X_1,X_2] &= \tau (X_1 + X_2) + X_3. \nonumber 
\end{align}
which implies that $\mathfrak{g}^{(1)} = \textrm{span} \{ X_1 + X_2 , 
X_3\}$ . Furthermore we have that 
\[\func{ad}_{X_1} | 
_{\mathfrak{g}^{(1)}} = \begin{pmatrix}
 \tau & 1 \\ 
1 & 0
\end{pmatrix}\]
relative to the basis $\{ X_1 + X_2 , X_3\}$. The 
characteristic polynomial  of $\func{ad}_{X_1} | 
_{\mathfrak{g}^{(1)}}$ 
is $t^2- \tau t-1$ and the eigenvalues are $( \tau \pm 
\sqrt{\tau^2+4})/2$. Following the classification procedure the algebra 
is 
$L(3,2,\frac{\lambda_1}{\lambda_2})=L(3,2,\frac{\lambda_2}{\lambda_1})$ 
where $\lambda_1$ and $\lambda_2$ are the eigenvalues.

In solution (2) the brackets are
\begin{align}
[X_1,X_3] &= X_1 - X_2\nonumber \\
[X_2,X_3] &= X_1 -  X_2  \label{alg2}\\
[X_1,X_2] &=  \tau (X_1 - X_2) + X_3.  \nonumber
\end{align} which implies that $\mathfrak{g}^{(1)} = \textrm{span} \{ 
X_1 
- X_2 , X_3\}$ . Furthermore we have that
\begin{align}
\func{ad}_{X_1} | _{\mathfrak{g}^{(1)}} = \begin{pmatrix}
- \tau & 1 \\ 
-1 & 0
\end{pmatrix} \label{adx12}
\end{align} relative to the basis $\{ X_1 - X_2 , X_3\}$. The 
characteristic polynomial  of $\func{ad}_{X_1} | 
_{\mathfrak{g}^{(1)}}$ 
is $t^2+ \tau t + 1$ and the eigenvalues are $(- \tau \pm 
\sqrt{ \tau^2-4})/2$. Following the classification procedure we 
get the following three possibilities:
\begin{enumerate}[(a)]
\item If $|\tau|= 2$ then the algebra is $L(3,3)$ since \eqref{adx12} 
does not diagonalises,
\item If $|\tau|> 2$ then the algebra is 
$L(3,2,\frac{\lambda_1}{\lambda_2})=L(3,2,\frac{\lambda_2}{\lambda_1})$ 
where $\lambda_1$ and $\lambda_2$ are the eigenvalues,
\item If $|\tau| < 2$ then the algebra is $L(3,4, 
|\tau|/\sqrt{4-\tau^2} )$.
\end{enumerate}

We remark that in solutions (1) and (2)(b) we do have distinct groups. 
Indeed suppose that 
\[\frac{ c_{12}^1 - \sqrt{(c_{12}^1)^2+4} } { 
c_{12}^1 + \sqrt{(c_{12}^1)^2+4} }= \left ( \frac{ -C_{12}^1 - 
\sqrt{(C_{12}^1)^2-4} } {-C_{12}^1 + \sqrt{(C_{12}^1)^2-4} } \right 
)^{\pm 1} \]
where  $c^k_{ij}$ denote the structure constant in solution (1) and 
$C^k_{ij}$ denote the structure constant in solution (2)(b). Then it 
follows that $c^1_{12}=C^1_{12} = 0$ which contradicts $|C^1_{12}|>2$.

In solution (3) the brackets are
\begin{align*}
[X_1,X_3] &= X_1 + (\kappa \mp 1) X_2 \\
[X_2,X_3] &= (\kappa \pm 1) X_1 - X_2 \\
[X_1,X_2] &= X_3. 
\end{align*} 
It follows that $\dim \mathfrak{g}^{(1)}<3$ if and only if  
$\kappa=0$ which reduces to particular cases of solutions (1) or (2).
If $\dim \mathfrak{g}^{(1)}=3$ the Killing form is 
\[  K= (2\kappa\mp 2) (x_1)^2 - 4 x_1 x_2 - (2\kappa\pm 2) (x_1)^2
+	2\kappa^2(x_3)^2.\] 
and so $\kappa\neq 0$ implies $\mathfrak{g} 
\simeq L(3,5)$.

\subsubsection{\textbf{Case} $\det h >0$} In this case we have $c=0$ 
and $ 
c_{13}^2 - c_{23}^1 = 2 \chi $, hence the brackets are
\begin{align*}
[X_1,X_3] &= c_{13}^2 X_2 \\
[X_2,X_3] &= (c_{13}^2-2\chi) X_1 \\
[X_1,X_2] &= c_{12}^1 X_1 + c_{12}^2 X_2 + X_3
\end{align*}
and
\begin{align*}
[X_1,[X_2,X_3]]+[[X_2,[X_3,X_1]]+& [X_3,[X_1,X_2]] =[X_3,[X_1,X_2]]\\
&= -c_{12}^1  c_{13}^2 X_2  -  c_{12}^2 (c_{13}^2-2 \chi) X_1. 
\end{align*}

The Jacobi identity implies that we also have the following equations:
\[c_{12}^1  c_{13}^2=0 \quad c_{12}^2 (c_{13}^2-2\chi)=0. \]
There are three possible solutions:
\begin{enumerate}
\item $c_{12}^1=0$,  $c_{12}^2=0$, $(\kappa=c_{13}^2-\chi)$
\item $c_{12}^1=0$,  $c_{13}^2-2 \chi=0$,  $(\kappa=(c_{12}^2)^2+\chi)$
\item $c_{13}^2=0$,  $c_{12}^2=0$,  $(\kappa=-(c_{12}^1)^2-\chi)$. 
\end{enumerate}

In solution (1) the brackets are 
\begin{align*}
[X_1,X_3] &= (\kappa + \chi) X_2 \\
[X_2,X_3] &= (\kappa - \chi) X_1 \\
[X_1,X_2] &= X_3. 
\end{align*}
We note that $\dim \mathfrak{g}^{(1)}=3$ if and only if $ 
\kappa^2-\chi^2 \ne 0$. The Killing form is
	\[  K= 2(\kappa+\chi)(x_1)^2 - 2(\kappa-\chi)(x_2)^2 + 
	2(\kappa^2-\chi^2)(x_3)^2.\] 
Hence we conclude that if $ \kappa+\chi <0$ and $\kappa-\chi>0$ then $K$ is negative definite and the algebra is $L(3,6)$ otherwise if  $ \kappa^2-\chi^2 
\neq 0$ the algebra is $ L(3,5)$. 

If $\chi=\kappa$ then $\mathfrak{g}^{(1)} = \textrm{span} \{ X_2 , 
X_3\}$. 
Furthermore we have that
\begin{align*}
\func{ad}_{X_1} | _{\mathfrak{g}^{(1)}} = \begin{pmatrix}
0 & 2 \kappa \\ 
1 & 0
\end{pmatrix} 
\end{align*} relative to the basis $\{ X_2 , X_3\}$. The characteristic 
polynomial  of $\func{ad}_{X_1} |_{\mathfrak{g}^{(1)}}$ is $t^2-2 
\kappa$. If $\kappa>0$  then the eigenvalues are $\pm \sqrt{2\kappa}$ 
and the classification procedure implies that the algebra is 
$L(3,2,-1)$. If $\kappa<0$  then the eigenvalues are $\pm i 
\sqrt{-2\kappa}$ and the classification procedure implies that the 
algebra is $L(3,4,0)$.

If $\chi=-\kappa$ then $\mathfrak{g}^{(1)} = \textrm{span} \{ X_1 , 
X_3\}$. Furthermore we have that
\begin{align*}
\func{ad}_{X_2} | _{\mathfrak{g}^{(1)}} = \begin{pmatrix}
0 & 2 \kappa \\ 
-1 & 0
\end{pmatrix}
\end{align*} relative to the basis $\{ X_1 , X_3\}$. The characteristic 
polynomial  of $\func{ad}_{X_1} |_{\mathfrak{g}^{(1)}}$ is $t^2+2 
\kappa$. If $\kappa>0$  then the eigenvalues are $\pm i \sqrt{2\kappa}$ 
and the classification procedure implies that the algebra is 
$L(3,4,0)$. If $\kappa<0$  then the eigenvalues are $\pm 
\sqrt{-2\kappa}$ and the classification procedure implies that the 
algebra is $L(3,2,-1)$.

In solution (2) the brackets are 
\begin{align*}
[X_1,X_3] &= 2\chi X_2 \\
[X_2,X_3] &= 0 \\
[X_1,X_2] &= c_{12}^2 X_2 + X_3, \quad ( (c_{12}^2)^2 =\kappa-\chi). 
\end{align*}
and we see that $\mathfrak{g}^{(1)} = \textrm{span} \{ X_2 , X_3\}$ . 
Furthermore,  we also have that 
\begin{align}
\func{ad}_{X_1} | _{\mathfrak{g}^{(1)}} = \begin{pmatrix}
c_{12}^2 & 2 \chi \\ 
1 & 0
\end{pmatrix} \label{adx1dethgr0}
\end{align}
relative to the basis $\{ X_2 , X_3\}$. The characteristic polynomial  
of $\func{ad}_{X_1} | _{\mathfrak{g}^{(1)}}$ is $t^2-c_{12}^2t-2 \chi$ 
and the eigenvalues are $(c_{12}^2 \pm \sqrt{(c_{12}^2)^2+8\chi})/2$. 
Following the classification procedure we get the following three 
possibilities:   
\begin{enumerate}[(a)]
\item If $(c_{12}^2)^2=-8\chi$ then the algebra is $L(3,3)$ since 
\eqref{adx1dethgr0} does not diagonalises.
\item If $(c_{12}^2)^2>-8\chi$ then the algebra is 
$L(3,2,\frac{\lambda_1}{\lambda_2})=L(3,2,\frac{\lambda_2}{\lambda_1})$ 
where $\lambda_i$ are the eigenvalues.
\item If $(c_{12}^2)^2<-8\chi$ then the algebra is $L(3,4, c_{12}^2/ 
\sqrt{-8\chi-(c_{12}^2)^2} )$.
\end{enumerate}

In solution (3) the brackets are 
\begin{align*}
[X_1,X_3] &= 0 \\
[X_2,X_3] &= -2\chi X_1 \\
[X_1,X_2] &= c_{12}^1 X_1 + X_3, \quad ( (c_{12}^1)^2 =-\kappa-\chi). 
\end{align*}
and we see that $\mathfrak{g}^{(1)} = \textrm{span} \{ X_1 , X_3\}$ . 
Furthermore,  we also have that
\begin{align}
\func{ad}_{X_2} | _{\mathfrak{g}^{(1)}} = \begin{pmatrix}
-c_{12}^1 & -2 \chi \\ 
-1 & 0
\end{pmatrix} \label{adx1dethgr02}
\end{align}
relative to the basis $\{ X_1 , X_3\}$. The characteristic polynomial  
of $\func{ad}_{X_1} | _{\mathfrak{g}^{(1)}}$ is $t^2+c_{12}^1t-2 \chi$ 
and the eigenvalues are $(-c_{12}^1 \pm \sqrt{(c_{12}^1)^2+8\chi})/2$. 
Following the classification procedure we get the following three 
possibilities:   
\begin{enumerate}[(a)]
\item If $(c_{12}^1)^2=-8\chi$ then the algebra is $L(3,3)$ since 
\eqref{adx1dethgr02} does not diagonalises.
\item If $(c_{12}^1)^2>-8\chi$ then the algebra is 
$L(3,2,\frac{\lambda_1}{\lambda_2})=L(3,2,\frac{\lambda_2}{\lambda_1})$ 
where $\lambda_i$ are the eigenvalues.
\item If $(c_{12}^1)^2<-8\chi$ then the algebra is $L(3,4, c_{12}^1/ 
\sqrt{-(c_{12}^1)^2-8\chi }  )$.
\end{enumerate}

We remark that the solutions (2)(b) and (3)(b) are  distinct except when $ c^1_{12}=\pm c^2_{12}$ and solutions (2)(c) and (3)(c) are  distinct except when $ c^1_{12}=\pm c^2_{12}$. In fact case (3) can be obtained from cases (2) by multiplying the metric by $-1$  (i.e., timelike becomes spacelike and vice versa).

\subsubsection{\textbf{Case} $\det h < 0$} In this case we have 
$c=\chi$ 
and $ 
c_{23}^1 =c_{13}^2$, hence the brackets are
\begin{align*}
[X_1,X_3] &= \chi X_1 + c_{13}^2 X_2 \\
[X_2,X_3] &= c_{13}^2 X_1 -  \chi X_2 \\
[X_1,X_2] &= c_{12}^1 X_1 + c_{12}^2 X_2 + X_3
\end{align*}
and
\begin{align*}
[X_1,[X_2,X_3]]+[[X_2,[X_3,X_1]]+& [X_3,[X_1,X_2]] =[X_3,[X_1,X_2]]\\
&= - (c^1_{12} \chi   +  c^2_{12} c^2_{13}) X_1 +(  c^2_{12} \chi-c^1_{12} c^2_{13}  ) X_2. 
\end{align*}

The Jacobi identity implies that we also have the following equations:
\[ c_{12}^1 \chi   +  c_{12}^2 c_{13}^2 =0 \quad \textrm{and} \quad 
c_{12}^2 \chi-c_{12}^1 c_{13}^2 =0.\]
It follows that $c^1_{12} = 0$,  $c^2_{12} = 0$ and $\kappa=c_{13}^2$,  
moreover the brackets are
\begin{align*}
[X_1,X_3] &= \chi X_1 + \kappa X_2, \\
[X_2,X_3] &= \kappa X_1 -  \chi X_2, \\
[X_1,X_2] &= X_3,
\end{align*}
implying $\dim \mathfrak{g}^{(1)}=3$. The Killing form is 
	\[  K= 2\kappa\left((x_1)^2 - (x_2)^2\right) - 4\chi x_1 x_2 
      +	2(\kappa^2+\chi^2)(x_3)^2.\] 
 It is sign-indefinite hence the algebra is $L(3,5)$. 

The classification is complete.

\end{document}